\documentclass[10pt,a4paper]{article}
\usepackage[utf8]{inputenc}
\usepackage{amsmath}
\usepackage{amsfonts}
\usepackage{amssymb}
\usepackage{amsthm}
\usepackage{color}

 \newtheorem{thm}{Theorem}[section]
 \newtheorem{cor}[thm]{Corollary}
 \newtheorem{lem}[thm]{Lemma}
 \newtheorem{prop}[thm]{Proposition}
 \theoremstyle{definition}
 \newtheorem{defn}[thm]{Definition}
 \theoremstyle{remark}
 \newtheorem{rem}[thm]{Remark}

 \numberwithin{equation}{section}

\newcommand{\vertiii}[1]{{\left\vert\kern-0.25ex\left\vert\kern-0.25ex\left\vert #1
  \right\vert\kern-0.25ex\right\vert\kern-0.25ex\right\vert}}

\newcommand{\Span}{\operatorname{span}}
\newcommand{\im}{\operatorname{Im}}
\newcommand{\re}{\operatorname{Re}}

\newcommand{\grass}{\operatorname{Graff}}

\begin{document}
\author{Wolfram Bauer
and Robert Fulsche
\footnote{Institut f\"{u}r Analysis, Welfengarten 1, 30167 Hannover, Germany.  {\bf Email:}  \texttt{bauer@math.uni-hannover.de} (W. Bauer) and 
\texttt{fulsche@math.uni-hannover.de} (R. Fulsche)}}
\title{Resolvent algebra in Fock-Bargmann representation}
\maketitle
\begin{abstract}\noindent
The resolvent algebra $\mathcal{R}(X, \sigma)$ associated to a symplectic space $(X, \sigma)$ was introduced by D.~Buchholz and H.~Grundling as a convenient model of the 
canonical commutation relation (CCR) in quantum mechanics. 
We first study a representation of $\mathcal{R}(\mathbb{C}^n, \sigma)$ with the standard symplectic form $\sigma$ inside the full Toeplitz algebra over the Fock-Bargmann space. We prove that $\mathcal{R}(\mathbb{C}^n, \sigma)$ itself is a Toeplitz algebra. In the sense of R. Werner's 
{\it correspondence theory} we determine its corresponding shift-invariant and closed space of symbols. Finally, we discuss a representation of the resolvent algebra $\mathcal{R}(\mathcal{H}, \tilde{\sigma})$ for an infinite dimensional symplectic separable Hilbert space $(\mathcal{H}, \tilde{\sigma})$. More precisely, we find a representation of $\mathcal{R}(\mathcal{H}, \tilde{\sigma})$ 
inside the full Toeplitz algebra over the Fock-Bargmann space in infinitely many variables. 
\vspace{1mm}\\
{\bf keywords:} canonical commutation relation, Toeplitz $C^*$ algebra, correspondence theorem, Gaussian measure on Hilbert space
\vspace{1mm}
\end{abstract}
\section{Introduction}
The classical CCR algebra (see \cite{Bratteli_Robinson2}) provides a standard $C^*$ algebraic model for the {\it canonical commutation relation} (CCR) in quantum mechanics. In typical applications, however, these algebras have significant disadvantages. In particular, CCR algebras in general do not contain bounded functions of the Hamiltonian (physical {\it observables}). Only in rare and physically less relevant cases they are stable under dynamics. More precisely, it was noticed in \cite{Fannes_Verbeure1974} that time evolution of the Hamiltonian $\mathbf{H}:=-\Delta + V$ on $L^2(\mathbb R^n)$ in the standard Schr\"{o}dinger representation gives rise to $\ast$-automorphisms of the CCR algebra over the standard symplectic space 
$\mathbb{R}^{2n}\cong T^*\mathbb{R}^n$ only in the trivial case of an identically vanishing potential function $V$. In order to resolve such problems, D. Buchholz and H. Grundling have proposed a new approach to CCR by introducing  the {\it resolvent algebra} $\mathcal{R}(X, \sigma)$ associated to a symplectic space $(X, \sigma)$ in \cite{Buchholz_Grundling2008}. This algebra abstractly is defined through certain algebraic relations that mimic properties of the resolvents of the 
(i.g. unbounded) {\it canonical operators} (see \eqref{canonical_operators}). 
\vspace{1ex}\par 
In the present paper, we first consider complex $n$-space $\mathbb{C}^n\cong \mathbb{R}^{2n}$ equipped with the standard symplectic structure. As is well-known, in this setting the CCR algebra is identified with the $C^*$ algebra generated by Toeplitz operators with symbols in the space $\textup{TP}$ of trigonometric polynomials. It is an interesting observation that this Toeplitz $C^*$ algebra coincides with the linear closure of Toeplitz operators with symbols in $\textup{TP}$ (see \cite{Coburn1999} for a precise statement and further extensions of these results). 
In a similar manner, our goal is to study the resolvent algebra in the Fock-Bargmann representation and to describe it as a concrete algebra generated by Toeplitz operators with shift-invariant symbol space. Being a shift-invariant  $C^*$ algebra, it has turned out that
a convenient mathematical framework for the analysis of the resolvent algebra is R. Werner's {\it quantum harmonic analysis} \cite{Werner1984} and its extension by R. Fulsche in \cite{Fulsche2020} to Toeplitz operator theory on the Fock-Bargmann space. A particularly useful tool is the {\it correspondence theorem}  (Theorem \ref{Correspondence_Theorem})
in \cite{Werner1984, Fulsche2020} which ensures the existence and uniqueness of a closed and shift-invariant space of bounded uniformly continuous functions on $\mathbb{C}^n$, which are the symbols of Toeplitz operators linearly generating the resolvent algebra.

 In dimension $n=1$ we show that the space corresponding to the resolvent algebra $\mathcal{R}(\mathbb{C}, \sigma)$ is the uniform closure of the classical resolvent functions. However, in higher dimensions $n>1$ we only can prove a weaker result (see Proposition \ref{Propositon_resolvent_algebra_is_TA}). Nevertheless, we show that $\mathcal{R}(\mathbb{C}^n, \sigma)$ coincides with a $C^*$ algebra generated by a concrete set of Toeplitz operators with bounded symbols. It remains an open problem whether or not the closure of classical resolvent functions 
 forms the corresponding space to the resolvent algebra in every complex dimension $n\in \mathbb{N}$ and in the sense of Theorem \ref{Correspondence_Theorem}. 
\vspace{1ex}\par 
In the second part of the paper, we replace $\mathbb{C}^n$ by a separable infinite dimensional complex symplectic Hilbert spaces $(\mathcal{H}, \tilde{\sigma})$. In \cite{Janas_Rudol1990,Janas_Rudol1995} the authors have proposed two definitions of Toeplitz operators on the infinitely many variable Fock-Bargmann space. Here, we follow the 
measure theoretical approach and consider $C^*$ algebras generated by Toeplitz operators over $\mathcal{H}$. Due to the non-nulcearity of $\mathcal{H}$ (as well as of the space of entire functions over $\mathcal{H}$ with the compact-open topology) as topological vector spaces and caused by features of the measure theory on infinite dimensional spaces, a variety of new effects can be observed in the theory of Toeplitz operators. In particular, in this setting a {\it correspondence theorem} so far is unknown, even though 
it may exist in a suitable formulation. Our main result of the last section (Theorem \ref{Theorem_Last_section}) states that there is a representation of the resolvent algebra corresponding to a symplectic Hilbert space $\mathcal{H}_{1/2}$ with Hilbert-Schmidt embedding $\mathcal{H}_{1/2} \hookrightarrow \mathcal{H}$ inside the full Toeplitz algebra over $\mathcal{H}$. 
\vspace{1ex} \par 
The structure of the paper is as follows: In Section \ref{Section_2} we recall the definition of the CCR and resolvent algebra \cite{Buchholz_Grundling2008}. We restate a well-known
representation of the CCR algebra associated to $\mathbb{C}^n$ with the standard symplectic structure as a $C^*$ algebra generated by Toeplitz operators.  

Section \ref{Section_Fock_Bargmann_space_and_TO} provides some basic facts on Fock-Bargmann spaces $F^2_{\bf t}$ over $\mathbb{C}^n$ and Toeplitz operators 
acting on $F^2_{\bf t}$. We define the Berezin transform and then express Weyl operators in form of Toeplitz operators with bounded symbols. Most of the material is standard and further 
details on the role of Toeplitz operators in quantum mechanics can be found in \cite{Berger_Coburn1987,Coburn1999}. 

Section \ref{Resolvent_algbera_FS} describes the Fock-Bargmann representation of the resolvent algebra. We present the {\it correspondence theorem} \cite{Werner1984} in this setting (see \cite{Fulsche2020}). In passing, we mention some applications to the analysis of Toeplitz operators on Fock-Bargmann spaces. We discuss the unique shift-invariant closed symbol space 
corresponding to the resolvent algebra. An essential ingredient to the analysis is an integral form of the Berezin transform on products of resolvents (Proposition \ref{formula:resolvent}).
Moreover, in our proofs we essentially use a specific compactification of the {\it affine Grassmannian} which was introduced in \cite{vanNuland}. 

Finally, in Section \ref{Infinite dimensional symplectic space}, we discuss the resolvent algebra associated to an infinite dimensional symplectic separable Hilbert space in the framework of 
Toeplitz operators on the Fock-Bargmann space of Gaussian square integrable entire functions in infinitely many variables. 
\vspace{1ex}\\
{\bf Acknowledgement:} We thank Prof. R. Werner for many enlightening discussions on {\it quantum harmonic analysis}, the {\it correspondence theorem} and the structure of the 
{\it resolvent algebra}. 
\section{CCR and Resolvent Algebra}
\label{Section_2}
Let $(X, \sigma)$ be a symplectic vector space. Consider a Hilbert space $\mathcal H$ and a real linear map $\phi$ from $(X, \sigma)$ into the space of (unbounded) self-adjoint operators 
on $\mathcal H$. Let us assume that all operators $\phi(f)$ are essentially self-adjoint on a common domain $\mathcal D$ which forms a core for each $\phi(f)$. Recall that the {\it canonical 
commutation relation} (CCR) have the form: 
\begin{equation}\label{canonical_operators}
\big{[}\phi(f), \phi(g)\big{]} = i\sigma(f, g), \quad f, g \in X. 
\end{equation}
We call $\phi(f)$ and $\phi(g)$ {\it canonical operators}. Especially in an algebraic setup the analysis of unbounded operators involves some difficulties  (compare e.g. \cite[Chapters VIII.5 and VIII.6]{Reed1}). Therefore, one may look for $C^*$ algebraic models which suitably encode (CCR). One standard idea consists in replacing $\phi(f)$ above by suitable bounded functions of $\phi(f)$.  
A classical approach due to H. Weyl amounts in considering the {\it CCR algebra} generated by the unitary operators $\exp(i\phi(f))$, $f \in X$. The modified CCR are:
\[ \exp\big{(}i\phi(f)\big{)} \exp\big{(}i\phi(g)\big{)} = e^{-i\sigma(f,g)} \exp\big{(}i\phi(f+g)\big{)} \]
and self-adjointness of $\phi(f)$ implies that $\exp(i\phi(f))^\ast = \exp(i\phi(-f))$. 
\vspace{1mm}\par 
 More abstractly, one is interested in the $C^\ast$ algebra $\textup{CCR}(X, \sigma)$ generated by the relations 
\begin{align*} 
W(f)W(g) &= e^{-i\sigma(f,g)} W(f+g), \hspace{5ex} f,g \in X\\
W(f)^\ast &= W(-f).
\end{align*}
Such {$C^\ast$ algebras} are usually called {\it CCR algebras}, see \cite{Bratteli_Robinson2}. 
\vspace{1mm}\par 
With $z,w \in \mathbb{C}^n$ put $\langle w,z\rangle:= w_1\overline{z}_1 + \ldots + w_n \overline{z}_n$  and consider the standard symplectic space $(\mathbb C^n, \sigma)$ 
with symplectic form
$$\sigma(w,z) := \im \langle w,z\rangle.$$   
One obtains a representation of $\textup{CCR}(\mathbb{C}^n, \sigma)$ as operators acting on the Fock-Bargmann space $F_1^2$ (see  \cite{Coburn1999} and Section \ref{Section_Fock_Bargmann_space_and_TO} for precise definitions). In fact, with our previous notation we put: 
\[ \phi(z) := T_{2\sigma(\cdot, z)},\]
where $T_{f}$ is a Toeplitz operator with complex valued symbol $f$ on $\mathbb{C}^n$.  
The corresponding {CCR algebra} is generated by the unitary {\it Weyl operators} (\ref{definition_Weyl_operator}) below which satisfy 
$$W_z= \exp\big{(}iT_{2\sigma(\cdot, z)}\big{)},$$ 
i.e. $iT_{2\sigma(\cdot, z)}$ is the generator of the unitary one-parameter group $(W_{tz})_{t \in \mathbb{R}}$. Hence, 
\begin{equation}\label{Fock_Bargmann_representation_CCR}
\textup{CCR}(\mathbb C^n, \sigma) \cong C^\ast\big{(} W_z \: : \: ~ z \in \mathbb C^n\big{)}. 
\end{equation}
\par 
The algebra $\textup{CCR}(\mathbb{C}^n, \sigma)$ is a classical object and in a more general setup it is discussed in \cite{Bratteli_Robinson2}.
In particular, the Fock-Bargmann representation of $\textup{CCR}(\mathbb C^n, \sigma)$ maps into the  {\it full Toeplitz algebra} $\mathcal{T}$, i.e. the $C^*$ algebra generated by 
Toeplitz operators with arbitrary bounded measurable symbols.  
\begin{thm}[{{\it L. A. Coburn }\cite{Coburn1999}}]
It holds
\[ \textup{CCR}(\mathbb C^n, \sigma) \cong \overline{\{ T_\phi \: : \: \phi \in TP\}}, \]
where {$\textup{TP}$ is the space of trigonometric polynomials:} 
\[ TP := \operatorname{span} \big{\{} w \mapsto \exp\big{(}i\sigma(w, z)\big{)}\: : \: z \in \mathbb C^n\big{\}} \subset L^\infty(\mathbb C^n).\]
Here we write {$\overline{\mathcal{M}}$ for the operator norm closure of a given set $\mathcal{M}\subset \mathcal{L}(F_1^2)$}. 
\end{thm}
As was mentioned in the introduction, there are certain drawbacks of using {\textup{CCR} algebras}  as quantum mechanical models. {A new algebraic framework was suggested in \cite{Buchholz_Grundling2008, Buchholz2014} which partly overcomes the above mentioned obstructions. Therein, the authors }consider the {\it resolvent algebra}, 
which is the $C^\ast$ algebra generated by the resolvents of the ({in general} unbounded) operators $\phi(f)$. 
\begin{defn}[{\it D. Buchholz, H. Grundling}  \cite{Buchholz_Grundling2008}] \label{definition_Resolvent_algebra}
Given a symplectic space $(X, \sigma)$ we define $\mathcal R_0(X, \sigma)$ as the universal unital $\ast$-algebra generated by the set 
$\{ R(\lambda, f) \:: \: ~\lambda \in \mathbb R \setminus \{0\}, \:  f \in X\}$ together with the relations: 
\begin{itemize}
\item[(1)] $R(\lambda, 0) = -\frac{i}{\lambda}1$, 
\item[(2)] $R(\lambda, f)^\ast = R(-\lambda, f)$, 
\item[(3)] $\nu R(\nu \lambda, \nu f) = R(\lambda, f)$, 
\item[(4)] $R(\lambda, f) - R(\mu, f) = i(\mu-\lambda) R(\lambda, f) R(\mu,f)$, 
\item[(5)] $[R(\lambda, f), R(\mu, g)] = i\sigma(f,g) R(\lambda, f) R(\mu,g)^2 R(\lambda,f)$, 
\item[(6)] $R(\lambda, f)R(\mu,g) = R(\lambda + \mu, f+g)\big{[}R(\lambda, f) + R(\mu,g) 
 + i\sigma(f,g)R(\lambda, f)^2 R(\mu,g)\big{]}$, 
\end{itemize}
for $\lambda, \mu, \nu \in \mathbb R \setminus \{ 0\}$ and $f, g \in X$.
\vspace{1ex}\par
The {\it resolvent algebra} $\mathcal R(X, \sigma)$ is defined as the closure of $\mathcal R_0(X, \sigma)$ with respect to a certain semi-norm
 obtained through the GNS construction (cf. \cite{Buchholz_Grundling2008}).
\end{defn}
\begin{rem}
\begin{enumerate}
\item {Note that Equations (1) and (2) encode $\phi(0) = 1$ and the self-adjointness of $\phi(f)$, respectively}. The third and sixth equation reflect 
the $\mathbb R$-linearity of $\phi$, whereas Equation (4) is the resolvent identity. Finally, the fifth equation is the substitute for the (CCR).
\item When one considers a representation of $\mathcal R_0(X, \sigma)$ as a concrete $\ast$-algebra generated by resolvents of self-adjoint operators on a 
Hilbert space, then $\mathcal R(X, \sigma)$ coincides with the operator norm closure of $\mathcal R_0(X, \sigma)$. 
\end{enumerate}
\end{rem}
In the present paper we study the {\it Fock-Bargmann representation} of the resolvent algebra over the standard symplectic space $\mathbb C^n\cong \mathbb{R}^{2n}$
as well as on an infinite dimensional Hilbert space with a symplectic structure. In the finite dimensional setting of $\mathbb{C}^n$ it  is well-known 
(cf. \cite{Berger_Coburn1987,Coburn1999} and Section \ref{Section_Fock_Bargmann_space_and_TO}) that the generators are  resolvents of self-adjoint Toeplitz operators $T_{2\sigma (\cdot, z)}$, $z \in \mathbb{C}^n$ defined below. 

Therefore, the resolvent algebra 
$\mathcal{R}(\mathbb{C}^n, \sigma)$ is the $C^\ast$ algebra generated by the resolvents of Toeplitz operators: 
\begin{equation} \label{Definition_resolvent_algebra}
\mathcal R(\mathbb C^n, \sigma) \cong C^\ast \Big{(} (i\lambda - T_{2\sigma(\cdot, z)})^{-1}; ~\lambda \in \mathbb R \setminus \{ 0\}, z \in \mathbb C^n \Big{)}. 
\end{equation}
In fact, this will be the starting point of our analysis.
\section{Fock-Bargmann space and Toeplitz operators}
\label{Section_Fock_Bargmann_space_and_TO}
Let $n \in\mathbb{N}$ and consider positive real numbers $t_j>0$ where $j=1, \ldots, n$. We write $\mathbf t := (t_1, \dots, t_n)$ and on $\mathbb C^n$ we define the 
probability measure $\mu_{\mathbf t}$ by
\[ d\mu_{\mathbf t}(z) = \frac{1}{\pi^n t_1 \cdot \dots \cdot t_n}e^{-\left (\frac{|z_1|^2}{t_1} + \dots + \frac{|z_n|^2}{t_n}\right )} dV(z), \hspace{3ex} z=(z_1, \ldots, z_n) \in \mathbb{C}^n.\]
Here, $V$ denotes the Lebesgue measure on $\mathbb C^n \cong \mathbb R^{2n}$. The Fock-Bargmann space $F_{\mathbf t}^2 = F_{\mathbf t}^2(\mathbb C^n)$ is defined as
\begin{equation}\label{definition_Fock_Bargmann_space}
F_{\mathbf t}^2 := L^2(\mathbb C^n, \mu_{\mathbf t}) \cap \operatorname{Hol}(\mathbb C^n),
\end{equation}
where $\operatorname{Hol}(\mathbb C^n)$ denotes the space of entire functions on $\mathbb C^n$. The reader experienced with the analysis on such spaces may wonder why we introduce the parameters $t_1, \ldots, t_n$ in the above definition, as the standard approach corresponds to the choice  $t_1 = t_2 = \dots = t_n = :t > 0$. We will use the notation $F_t^2 = F_{(t, \dots, t)}^2$ when we are in this standard situation. Conceptually, the more general setup does not cause any problems (apart from longer notations), but will be convenient when we pass to infinite dimensional symplectic spaces. 
\vspace{1mm}\par 
Throughout the paper we denote the standard inner product of $L^2(\mathbb C^n, \mu_{\mathbf t})$ and of $F_{\mathbf t}^2$ by 
\[  \langle f,g\rangle:= \int_{\mathbb C^n} f(z) \overline{g(z)}d\mu_{\mathbf t}(z) \]
and we write $\| f\| := \langle f, f\rangle^{1/2}$ for the induced norm. As is well-known,  $F_{\mathbf t}^2$ is a reproducing kernel Hilbert space with kernel function
\[ K_w^{\mathbf t}(z) := K^{\mathbf t}(z,w) := e^{\frac{z_1 \cdot \overline{w_1}}{t_1} + \dots + \frac{z_n \cdot \overline{w_n}}{t_n}}, \] 
where $w=(w_1, \ldots, w_n) \in \mathbb{C}^n.$ In order to simplify notations, we write
\begin{align}
    \langle z, w\rangle_{\mathbf t}:= \frac{z_1 \cdot \overline{w_1}}{t_1} + \dots + \frac{z_n \cdot \overline{w_n}}{t_n} 
    \hspace{3ex} \mbox{\it and } \hspace{3ex} \|z\|_{\mathbf t}^2:= \langle z,z \rangle_{\mathbf t}
\end{align}
such that
 $  K_w^{\mathbf t}(z) = e^{\langle z, w\rangle_{\mathbf t}}$.
 We express the {\it normalized reproducing kernels} $k_w^{\mathbf t} \in F_{\mathbf t}^2$ as 
\[ k_w^{\mathbf t}(z) := \frac{K_w^{\mathbf t}(z)}{\| K_w^{\mathbf t}\|} = e^{\langle z, w\rangle_{\mathbf t} - \frac{1}{2} \| w\|_{\mathbf t}^2}. \]
\par 
Let $A \in \mathcal L(F_{\mathbf t}^2)$ be a bounded linear operator on $F_{\mathbf t}^2$. We define the {\it Berezin transform} $\widetilde{A}$ of $A$ 
by 
\[ \widetilde{A}(z) :=\big{\langle} A k_z^{\mathbf t}, k_z^{\mathbf t} \big{\rangle}, \hspace{4ex} z \in \mathbb{C}^n. \]
Note that $\widetilde{A}(z)$ is a bounded real-analytic function. Moreover, the map $A \mapsto \widetilde{A}$ is known to be injective, \cite{Folland1989}. 
\vspace{1ex}\par 
As $F_{\mathbf t}^2$ is a closed subspace of $L^2(\mathbb C^n, \mu_{\mathbf t})$, there is an orthogonal projection $P^{\mathbf t}: 
L^2(\mathbb C^n, \mu_{\mathbf t}) \to F_{\mathbf t}^2$ acting on $ f \in L^2(\mathbb C^n, \mu_{\mathbf t})$ as 
\[ P^{\mathbf t}(f)(w) = \big{\langle} f, K_w^{\mathbf t}\big{\rangle}, \quad w \in \mathbb C^n. \]
For a measurable function $\varphi: \mathbb C^n \to \mathbb C$, we let $T_{\varphi}^{\mathbf t}$ denote the {\it Toeplitz operator} with symbol $\varphi$, which is defined by
\[ T_{\varphi}^{\mathbf t}(g) := P^{\mathbf t}(\varphi g) \]
on the natural domain
$D(T_{\varphi}^{\mathbf t}) = \big{\{}g \in F_{\mathbf t}^2 \: : \:  \varphi g \in L^2(\mathbb C^n, \mu_{\mathbf t})\big{\}}$. 
\vspace{1ex} \par 
If $\varphi \in L^\infty(\mathbb C^n)$, then $T_{\varphi}^{\mathbf t}$ is a bounded operator. For a given set $S \subset L^\infty(\mathbb C^n)$, we will denote by $\mathcal T^{\mathbf t}(S)$ the $C^\ast$ algebra generated by all Toeplitz operators with symbols in $S$ and set 
$$\mathcal T^{\mathbf t} := \mathcal T^{\mathbf t} \big{(}L^\infty(\mathbb C^n)\big{)}$$ 
for the {\it full Toeplitz algebra}. We also define $\mathcal T_{\textup{lin}}^{\mathbf t}(S)$ to be the closed linear span of Toeplitz operators with symbols in $S$.
\vspace{1ex} \par 
Let us introduce {\it Weyl operators} on the Fock-Bargmann space $F_{\mathbf t}^2$: if $z \in \mathbb C^n$ then we define $W_z^{\mathbf t} \in \mathcal L(F_{\mathbf t}^2)$ by
\begin{equation}\label{definition_Weyl_operator}
W_z^{\mathbf t}(g)(w) = k_z^{\mathbf t}(w) g(w-z). 
\end{equation}
Weyl operators are well-known to be unitary. Moreover, they fulfill the relation
\[ \big{(}W_z^{\mathbf t}\big{)}^\ast = \big{(}W_z^{\mathbf t}\big{)}^{-1} = W_{-z}^{\mathbf t} \quad  \mbox{\it and} \quad 
W_z^{\mathbf t} W_w^{\mathbf t} = e^{-i\sigma_{\mathbf t}(z, w)}W_{z+w}^{\mathbf t}. \]
Here, the symplectic form $\sigma_{\mathbf t}$ on $\mathbb C^n$ with parameter ${\bf t}$ is defined by 
\begin{align*}
    \sigma_{\mathbf t}(w, z) := \frac{\im(w_1 \cdot \overline{z_1})}{t_1} + \dots + \frac{\im(w_n \cdot \overline{z_n})}{t_n} = \im \langle w, z\rangle_{\mathbf t}.
\end{align*}
As a matter of fact, Weyl operators are Toeplitz operators themselves. If we define the family $(g_z^{\bf t})_{z\in \mathbb{C}^n}$ of bounded functions on $\mathbb{C}^n$ by 
\[ g_z^{\mathbf t} (w) := e^{\frac{1}{2}\| z\|_{\mathbf t}^2 + 2i\sigma_{\mathbf t}(w,z)}, \]
then it holds
\begin{equation}\label{Relation_Toeplitz_Weyl_operator}
W_z^{\mathbf t} = T_{g_z^{\mathbf t}}^{\mathbf t}. 
\end{equation}
In fact,  (\ref{Relation_Toeplitz_Weyl_operator}) follows by showing that the Berezin transform of both sides coincide 
(see \cite{Berger_Coburn1987,Coburn1999} and the references therein).
\section{Resolvent algebra in Fock-Bargmann \\representation}
\label{Resolvent_algbera_FS}
In this section, we study the resolvent algebra $\mathcal R(\mathbb C^n, \sigma_{\mathbf t})$ in its Fock-Bargmann representation. Of course, it is not hard to reduce the analysis 
of $\mathcal R(\mathbb C^n, \sigma_{\mathbf t})$ to that of $\mathcal R(\mathbb C^n, \sigma)$. Again, we emphazise that the extra flexibility coming from the parameter set 
${\bf t}=(t_1, \ldots, t_n)$ will be useful when passing to the infinite dimensional limit $n \rightarrow \infty$. Therefore, we take the (mostly notational) burden upon us to carry $\mathbf t$ all the way through this section. For readability, we will denote the Fock-Bargmann representation of the resolvent algebra also by $\mathcal R(\mathbb C^n, \sigma_{\mathbf t})$. 
Since the Toeplitz operators $T_{2\sigma_{\mathbf t} (\cdot, z)}^{\mathbf t}$ fulfill the canonical commutation relation (equivalently, the Weyl operators $W_z^{\mathbf t}$ fulfill the relation of the CCR algebra), the resolvent algebra is the $C^\ast$ algebra generated by resolvents of Toeplitz operators. The following integral representations allow us to study $\mathcal R(\mathbb C^n, \sigma_{\mathbf t})$ more in detail:
\begin{lem}
For $z \in \mathbb C^n \setminus \{0\}$ the map $\mathbb{R}\ni s \mapsto W_{sz}^{\mathbf t}$ defines a strongly continuous unitary one-parameter group with generator $iT_{2\sigma_{\mathbf t} ( \cdot, z)}^{\mathbf t}$. In particular, for $\lambda > 0$ the following integral representations of the resolvents hold in the strong sense:
\begin{equation}\label{resolventrepr}
\big{(}T_{2\sigma_{\mathbf t} (\cdot, z)}^{\mathbf t} + i\lambda\big{)}^{-1} = \big{(}T_{2\sigma_{\mathbf t}( \cdot, z)+ i \lambda}^{\mathbf t}\big{)}^{-1} = -i \int_0^\infty e^{-\lambda s} W_{sz}^{\mathbf t} ds
\end{equation}
and
\begin{equation}\label{resolventrepr2}
\big{(}T_{2\sigma_{\mathbf t} ( \cdot, z)}^{\mathbf t} - i\lambda\big{)}^{-1} = \big{(}T_{2\sigma_{\mathbf t}( \cdot, z)- i \lambda}^{\mathbf t}\big{)}^{-1} = i \int_0^\infty e^{-\lambda s} W_{-sz}^{\mathbf t} ds. 
\end{equation}
\end{lem}
\begin{proof}
It is easy to check that $\mathbb{R} \ni s \mapsto W_{sz}^{\mathbf t}$ is indeed a strongly continuous unitary one-parameter group. For determining its generator we compute
\[ \lim_{s \to 0} \frac{W_{sz}^{\mathbf t}f - f}{s} = \lim_{s \to 0} \frac{T_{g_{sz}^{\mathbf t}}^{\mathbf t}f - f}{s} = \lim_{s \to 0} T_{(g_{sz}^{\mathbf t} - 1)/s}^{\mathbf t}f \]
for all $f \in F_{\mathbf t}^2$ such that the limit exists. 
{According to Stone's Theorem there is} a unique self-adjoint operator $A$ with domain $D(A)$ such that $iA$ generates the  group, i.e. for $f\in D(A)$ it holds
\begin{equation}\label{generatorlimit}
 \lim_{s \to 0} T_{(g_{sz}^{\mathbf t}-1)/s}^{\mathbf t}f = iAf
\end{equation}
in the norm sense. Since norm convergence in the reproducing kernel Hilbert space $F_{\mathbf t}^2$ implies pointwise convergence, it suffices to determine the pointwise limit of  
the right hand side of (\ref{generatorlimit}). Let $u \in \mathbb C^n$ and note that: 
\begin{align*}
\frac{T_{g_{sz}^{\mathbf t}-1}^{\mathbf t}f}{s}(u) 
=\int_{\mathbb C^n} \frac{1}{s}\left ( e^{\frac{s^2}{2} \|z\|^2_{\mathbf t} + 2is \sigma_{\mathbf t}(w, z)} - 1 \right )e^{\langle u,w\rangle_{\mathbf t}} f(w) d\mu_{\mathbf t}(w). 
\end{align*}
The integrand converges as $s \rightarrow 0$: 
\begin{align*}
\frac{1}{s}\left( e^{\frac{s^2}{2} \|z\|_{\mathbf t}^2 + 2is\sigma_{\mathbf t}(w,z)} - 1\right) \overset{s \to 0}{\longrightarrow} 2i\sigma_{\mathbf t}(w,z).
\end{align*}
An easy application of the dominated convergence theorem yields
\[ T_{(g_{sz}^{\mathbf t}-1)/s}^{\mathbf t}f(u) \to iT_{2\sigma_{\mathbf t}(\cdot, z)}^{\mathbf t} f(u) \]
for those $f \in F_{\mathbf t}^2$ such that $w \mapsto \sigma_{\mathbf t}(w ,z)f(w)\in L^2(\mathbb C^n, \mu_{\mathbf t})$.
Hence, 
 $$T_{(g_{sz}^{\mathbf t}-1)/s)}^{\mathbf t}f \to iT_{2\sigma_{\mathbf t}(\cdot, z)}^{\mathbf t}f=iAf \hspace{3ex} \mbox{\it  for} \hspace{4ex} f \in D(A).$$
Therefore, $iT_{2\sigma_{\mathbf t}( \cdot, z)}^{\mathbf t}$ is the generator of the one-parameter group $s \mapsto W_{sz}^{\mathbf t}$. Since all operators 
$W_{sz}^{\mathbf t}$ are unitary, the group $(W_{sz}^{\mathbf t})_s$ has growth bound $\omega_0 = 0$. For $\lambda > 0$ it follows (cf. \cite[Theorem I.1.10]{Engel_Nagel}) 
\[ \big{(}\lambda - iT_{2\sigma_{\mathbf t}(\cdot, z)}^{\mathbf t}\big{)}^{-1} = \int_0^\infty e^{-\lambda s} W_{sz}^{\mathbf t} ds \]
strongly, i.e.
\[ \big{(}T_{2\sigma_{\mathbf t}( \cdot, z)}^{\mathbf t} + i\lambda\big{)}^{-1} = -i\int_0^\infty e^{-\lambda s} W_{sz}^{\mathbf t} ds. \]
The integral representation (\ref{resolventrepr2}) follows from:
\begin{align*}
\big{(}T_{2\sigma_{\mathbf t}(\cdot, z)}^{\mathbf t} - i\lambda\big{)}^{-1} &= \left( (T_{2\sigma_{\mathbf t}(\cdot, z)}^{\mathbf t} + i\lambda)^{-1}\right)^\ast = i\int_0^\infty e^{-\lambda s} W_{-sz}^{\mathbf t} ds, 
\end{align*}
which completes the proof. 
\end{proof}
\begin{rem}
An inspection of the above arguments shows that the limit in Equation \eqref{generatorlimit} exists for all $f = K_w^{\mathbf t}, ~w \in \mathbb C^n$. Moreover, one easily verifies that 
the space $\mathcal{X}:=\operatorname{span}\{K_w^{\mathbf t} \: : \: w\in \mathbb C^n\}$ is invariant under the action of $W_{sz}^{\mathbf t}$. Then, \cite[Theorem VIII.10]{Reed1} implies that $T_{2\sigma_{\mathbf t}(\cdot, z)}^{\mathbf t}$ is essentially self-adjoint on $\mathcal{X}$.
\end{rem}
Define for $z \in \mathbb C^n, ~\lambda \in \mathbb C \setminus i\mathbb R$:
\[ R_{\mathbf t}(\lambda, z) := \big{(}T_{2\sigma_{\mathbf t}(\cdot, z)}^{\mathbf t} - i\lambda\big{)}^{-1}. \]
For simplicity, we will occasionally suppress $\mathbf t$ in the notation and  shortly write $R(\lambda, z) = R_{\mathbf t}(\lambda, z)$. Before we continue with our investigation, we need to recall the following well-known expansion of the resolvent. 
\begin{lem}\label{Resolvent_Neumann_series}
Let $\lambda_0, \lambda \in \mathbb C \setminus i\mathbb R$ such that $|\lambda_0 - \lambda| < |\lambda_0|$. Then, it holds
\begin{equation}\label{GL_von_Neumann_series_expansion_resolvent}
R(\lambda, z) = \sum_{k=0}^\infty (\lambda-\lambda_0)^k i^k R(\lambda_0, z)^{k+1},
\end{equation}
where the series converges in operator norm. In particular:
\begin{equation}\label{representation_power_of_a_resolvent}
R(\lambda_0,z)^k=\frac{i^{k-1}}{(k-1)!} \frac{{\rm d}^{k-1}}{{\rm d}\lambda^{k-1}}|_{\lambda=\lambda_0} R(\lambda, z).
\end{equation}
\end{lem}
The previous lemma has the following important consequence:
\begin{cor}
$\mathcal R(\mathbb C^n, \sigma_{\mathbf t}) = C^\ast( R_{\mathbf t}(\lambda, z): ~\lambda \in \mathbb C \setminus i\mathbb R, ~z \in \mathbb C^n).$
\end{cor}
We further note that Equations \eqref{resolventrepr} and \eqref{resolventrepr2} extend to $\lambda \in \mathbb C \setminus i\mathbb R$ with the same proof as in the case $\lambda \in \mathbb R \setminus \{ 0\}$: 
\begin{equation}\label{resolventrepr3}
R(\lambda, z) = i\int_0^\infty e^{-\lambda s} W_{-sz}^{\mathbf t} ds, \quad \re(\lambda) > 0
\end{equation}
and
\begin{equation}\label{resolventrepr4}
R(\lambda, z) = -i\int_0^\infty e^{\lambda s} W_{sz}^{\mathbf t} ds, \quad \re(\lambda) < 0.
\end{equation}
We mention two ways of analyzing  the connection between the algebra $\mathcal R(\mathbb C^n, \sigma_{\mathbf t})$ in its Fock-Bargmann space representation and the theory of Toeplitz operators and, more precisely, its realization as a Toeplitz algebra. The computational-heavy method studies the resolvents 
through their Laplace transform representation, using that the Weyl operators themselves are Toeplitz operators. Alternatively, one can use ``soft analysis'' arguments 
from the theory of Toeplitz operators, most prominently those arising in the theory of {\it quantum harmonic analysis} \cite{Werner1984, Fulsche2020}. Both approaches 
have their advantages and so we will try to shed some light on either of them. 
\subsection{Correspondence Theory in the Fock-Bargmann space}
We recall some aspects of {\it quantum harmonic analysis} in the setting of the Fock-Bargmann space and explain a particular part of it: the \emph{correspondence theory}. We 
refer to \cite{Werner1984,Fulsche2020} for further details. 
Roughly speaking, \emph{correspondence theory} relates certain subspaces of $\operatorname{BUC}(\mathbb C^n)$, the $C^\ast$ algebra of bounded, uniformly continuous functions on $\mathbb{C}^n$, 
with corresponding subspaces of $\mathcal L(F_{\mathbf t}^2)$. 
We start with some notation. Let $f: \mathbb C^n \to \mathbb C$ be a function and $z \in \mathbb C^n$, then put: 
\begin{align*}
\alpha_z(f)(w) := f(w-z).
\end{align*}
Given an operator $A \in \mathcal L(F_{\mathbf t}^2)$ we {define its shift by $z \in \mathbb C^n$ through}
\begin{align*}
\alpha_z(A) = W_z^{\mathbf t} A W_{-z}^{\mathbf t} = W_z^{\mathbf t} A (W_z^{\mathbf t})^\ast,
\end{align*}
where $W_z^{\mathbf t}$ is a Weyl operator. Clearly, $\operatorname{BUC}(\mathbb C^n)$ is the subalgebra of $L^\infty(\mathbb C^n)$ consisting of functions 
$f$ for which $z \mapsto \alpha_z(f)$ is continuous with respect to the $L^{\infty}$-norm. The analogous space on the operator side is
\begin{align*}
\mathcal C_1^{\mathbf t} = \big{\{} A \in \mathcal L(F_{\mathbf t}^2): ~z \mapsto \alpha_z(A) \text{ is } \| \cdot\|_{op}\text{-\it continuous}\big{\}}.
\end{align*}
\par 
The Fock-Bargmann space formulation of the correspondence theorem  due to R. Werner in  \cite{Werner1984} specifically 
tailored for Toeplitz operators can be found in \cite{Fulsche2020} within the standard situation $t_1 = t_2 = \dots = t_n = t > 0$. The case 
of positive weight parameters $t_1, \ldots, t_n$ in the definition of the Gaussian measure $\mu_{\mathbf t}$ follows by identical proofs and obvious modifications:
\begin{thm}[Correspondence Theorem, \cite{Werner1984, Fulsche2020}]\label{Correspondence_Theorem}
Let $\mathcal D_1 \subset \mathcal C_1^{\mathbf t}$ be a closed, $\alpha$-invariant subspace (meaning $\alpha_z(A) \in \mathcal D_1$ for every $z \in \mathbb C^n, ~A\in \mathcal D_1$). Then, there is a unique $\alpha$-invariant closed subspace $\mathcal D_0$ of $\operatorname{BUC}(\mathbb C^n)$ such that
\begin{align*}
\mathcal D_1 = \mathcal {T}_{lin}^{\mathbf t}(\mathcal D_0).
\end{align*}
Further, for an operator $A \in \mathcal C_1^{\mathbf t}$ the following statements are equivalent:
\begin{align*}
A \in \mathcal D_1 \Longleftrightarrow \widetilde{A} \in \mathcal D_0.
\end{align*}
Finally, $\mathcal D_0$ can be computed as follows:
$\mathcal D_0 = \overline{\{ \widetilde{A}: ~A \in \mathcal D_1\}}.$ In the following we call $\mathcal{D}_0$ and $\mathcal{D}_1$ corresponding spaces. 
\end{thm}
Here, we used the notation
\begin{align*}
 \mathcal T_{lin}^{\mathbf t}(\mathcal D_0) = \overline{\big{\{} T_f^{\mathbf t}:~f \in \mathcal D_0\big{\}}}.
\end{align*}
The previous result is a useful tool in the study of Toeplitz operators and Toeplitz algebras when it is combined with the following:
\begin{thm}[\cite{Fulsche2020}] 
The following $C^*$ algebras coincide: $\mathcal C_1^{\mathbf t} = \mathcal T^{\mathbf t}$.
\end{thm}
We mention that the previous two results extend to the case of the $p$-Fock space $F_{\bf t}^p$ where $p \neq 2$, cf. \cite{Fulsche2020, Fulsche2021}. 
They have immediate applications such as simple proofs of a compactness characterization for operators acting on the Fock-Bargmann space. We only present the 
case $p=2$. 
\vspace{1mm}\par 
Let $\mathcal{K}(\mathcal H)$ denote the ideal of compact operators on a Hilbert space $\mathcal H$. 
\begin{thm}[\cite{Bauer_Isralowitz2012}]\label{thm:compactness}
Let $A \in \mathcal L(F_{\mathbf t}^2)$. Then, the following are equivalent:
\begin{align*}
A \in \mathcal K(F_{\mathbf t}^2) \Longleftrightarrow A \in \mathcal T^{\mathbf t} \text{ and } \widetilde{A} \in C_0(\mathbb C^n). 
\end{align*}
Here, $C_0(\mathbb{C}^n)$ denotes the space of continuous complex valued functions vanishing at infinity. 
\end{thm}
The short proofs of the last theorems based on correspondence theory can be found in \cite{Fulsche2020}. We now want to demonstrate how Theorem \ref{Correspondence_Theorem} 
can be applied in order to gain a better understanding of the resolvent algebra. First, we note:
\begin{lem}\label{lemma:shiftofresolvent}
Let $z, w \in \mathbb C^n$ and $\lambda \in \mathbb C \setminus i\mathbb R$. Then: 
\begin{align*}
\alpha_w\big{(}R(\lambda, z)\big{)} = R\big{(}\lambda + 2i\sigma_{\mathbf t}(z,w), z\big{)}.
\end{align*}
\end{lem}
\begin{proof}
Applying the CCR of Weyl operators, we get for $\re(\lambda) > 0$:
\begin{align*}
\alpha_w\big{(}R(\lambda, z)\big{)} &= W_w^{\mathbf t} i\int_0^\infty e^{-\lambda s} W_{-sz}^{\mathbf t}~ds W_{-w}^{\mathbf t}\\
&= i\int_0^\infty e^{-\lambda s} W_w^{\mathbf t} W_{-sz}^{\mathbf t} W_{-w}^{\mathbf t}~ds\\
&= i\int_0^\infty e^{-\lambda t}  e^{-2is\sigma_{\mathbf t}(z,w)} W_{-sz}^{\mathbf t}~ds\\
&= R\big{(}\lambda+2i\sigma_{\mathbf t}(z,w), z\big{)}.
\end{align*}
The case $\re(\lambda) < 0$ follows analogously.
\end{proof}
\begin{prop}\label{Proposition_alpha_invariance}
$\mathcal R(\mathbb C^n, \sigma_{\mathbf t})$ is a closed, $\alpha$-invariant subspace of $\mathcal C_1^{\mathbf t}$.
\end{prop}
\begin{proof}
The $\alpha$-invariance immediately follows from the previous lemma. Further, the resolvent algebra is closed by definition. 
We only need to prove that it is contained in $\mathcal C_1^{\mathbf t}$. It suffices to show that $R_{\mathbf t}(\lambda, z) \in \mathcal C_1^{\mathbf t}$ for $\lambda \in \mathbb C \setminus i\mathbb R$ and $z \in \mathbb C^n$. This is an easy consequence of Lemma \ref{Resolvent_Neumann_series} and Lemma \ref{lemma:shiftofresolvent}. In fact, for $|w|$ sufficiently small such that simultaneously
$$2|\sigma_{\mathbf t}(z,w)| < |\lambda| \hspace{3ex} \mbox{\it and } \hspace{3ex} |2\sigma_{\mathbf t}(z,w)| \| R(\lambda, z)\| < 1$$
it follows: 
\begin{align*}
\| R(\lambda, z) - \alpha_w(R(\lambda, z))\| &= \| R(\lambda, z) - R(\lambda + 2i\sigma_{\mathbf t}(z, w), z)\|\\
&\leq \sum_{k=1}^\infty |2\sigma_{\mathbf t}(z,w)|^k \| R(\lambda, z)\|^{k+1}\\
&= \frac{2|\sigma_{\mathbf t}(z,w)| \| R(\lambda, z)\|^2}{1 - 2|\sigma_{\mathbf t}(z,w)|\| R(\lambda, z)\|}.
\end{align*}
This last expression tends to $0$ as $|w| \to 0$.
\end{proof}
Proposition \ref{Proposition_alpha_invariance} and the correspondence theorem (Theorem \ref{Correspondence_Theorem}) imply that 
there is a closed, $\alpha$-invariant subspace $\mathcal D_0^{\mathbf t}$ of $\operatorname{BUC}(\mathbb C^n)$ (possibly depending on the parameter tuple $\mathbf t$) such that
\begin{align*}
\mathcal R(\mathbb C^n, \sigma_{\mathbf t}) = \mathcal T_{\textup{lin}}^{\mathbf t}(\mathcal D_0^{\mathbf t}).
\end{align*}
Since $\mathcal R(\mathbb C^n, \sigma_{\mathbf t})$ itself is a $C^\ast$ algebra, we get:
\begin{thm}
$\mathcal R(\mathbb C^n, \sigma_{\mathbf t})$ is a Toeplitz algebra. More precisely, 
\begin{align*}
\mathcal R(\mathbb C^n, \sigma_{\mathbf t}) = \mathcal T^{\mathbf t}(\mathcal D_0^{\mathbf t}) := C^\ast(T_f^{\mathbf t}: ~f \in \mathcal D_0^{\mathbf t}) 
\end{align*}
for a suitable $\alpha$-invariant subspace $\mathcal{D}_0^{\mathbf t} \subset \textup{BUC}(\mathbb{C}^n)$. 
\end{thm}
Our next task is to determine the space $\mathcal D_0^{\mathbf t}$ explicitly.
\subsection{Computing $\mathcal D_0^{\mathbf t}$}
{\it Correspondence theory} has the nice flavour that in several examples the corresponding spaces (in the sense of Theorem \ref{Correspondence_Theorem}) are what 
one might naively expect; for example, $C_0(\mathbb C^n)$ corresponds to the compact operators $\mathcal K(F_{\mathbf t}^2)$ and the almost periodic functions 
correspond to the CCR algebra. In this respect, one might hope that the space $\mathcal D_0^{\mathbf t}$ corresponding to $\mathcal R(\mathbb C^n, \sigma_{{\bf t}})$ is:
\begin{align*}
 {\mathcal R_{cl} }= C^\ast\Big{(}\{ (\lambda - 2i\sigma_{\mathbf t}(\cdot, z))^{-1} \: : \:  ~\lambda \in \mathbb C \setminus i\mathbb R, ~z \in \mathbb C^n\}\Big{)}. 
\end{align*}
Note that $\mathcal R_{cl}$ does not depend on the choice of the parameter set ${\bf t}=(t_1, \ldots, t_n)$ and we can replace $\sigma_{\bf t}$ by $\sigma=\sigma_1$ in its definition 
(the ``contribution'' of the parameters $t_k$ can be absorbed into $z$). 
\vspace{1ex}\par
Since $\mathcal R_{cl}$ is indeed a subalgebra of $\operatorname{BUC}(\mathbb C^n)$, which is further $\alpha$-invariant as well as invariant under the parity operation 
$f\mapsto f(-\cdot),$ Theorem 3.13 in
\cite{Fulsche2020} implies that: 
$${ \mathcal T_{lin}^{\mathbf t}(\mathcal R_{\textup{cl}}) = \mathcal T^{\mathbf t}(\mathcal R_{cl})}.$$ 
Therefore, $\mathcal D_0^{\mathbf t} = \mathcal R_{cl}$ if and only if $\mathcal R(\mathbb C^n, \sigma_{\mathbf t}) = \mathcal T^{\mathbf t}(\mathcal R_{cl})$.
Next, we collect some useful facts:
\begin{prop}\label{Proposition_inclusion_of_spaces}
The following inclusions hold:
\begin{enumerate}
\item $C_0(\mathbb C^n) \subset \mathcal R_{cl}$.
\item $\mathcal K(F^2_{\mathbf t}) \subset \mathcal R(\mathbb C^n, \sigma_{\mathbf t})$.
\end{enumerate}
\end{prop}
\begin{proof}The first statement is an easy application of the Stone-Weierstrass Theorem; the second statement is \cite[Theorem 5.4]{Buchholz_Grundling2008}.\end{proof}
In what follows some computations can no longer be avoided. The Berezin transform of a product of resolvents can be computed explicitly, which turns out to be quite useful. We will not need the formula in full generality, but still provide the complete expression here.
\begin{prop}\label{formula:resolvent}
Let $m\in \mathbb{N}$ and $z_1, \dots, z_m \in \mathbb C^n$. Given a multi-index ${\bf k}=(k_1, \dots, k_m) \in {\mathbb N}^m$ and $\lambda_1, \dots, \lambda_m \in \mathbb C \setminus i\mathbb R$ 
we have:
\begin{align*}
& \Big{(}R(\lambda_1, z_1)^{k_1} \dots R(\lambda_m, z_m)^{k_m}\Big{)}^{\sim}(w) =\\
&\quad =
\frac{Ci^{|\mathbf k|}}{(\mathbf{k-1})!}\int_{(0,\infty)^m} \mathbf s^{\mathbf{k-1}}e^{-\Lambda \cdot \mathbf s - 2i\sigma_{\mathbf t}(w, \mathbf s \cdot \mathbf z) - i\sum_{j < \ell} s_j s_\ell \sigma_{\mathbf t}(z_j, z_\ell) - \frac{1}{2}\|\mathbf s \cdot \mathbf z\|_{\mathbf t}^2} ~d\mathbf s, 
\end{align*}
where $$C=C_{k, \lambda,m}:= (-1)^{|\mathbf k| - m} \cdot  \prod_{j=1}^m \operatorname{sign}\big{(}\re(\lambda_j)\big{)}^{k_j}\in\{ -1,1\}.$$ 
Here, we have used the standard multi-index notation together with: 
\begin{align*}
 \mathbf{1} &:= (1, \dots, 1)\in \mathbb{N}^m,\\
 \Lambda &:= \big{(}\operatorname{sign}(\re(\lambda_1))\lambda_1, \dots, \operatorname{sign}(\re(\lambda_m))\lambda_m\big{)},\\
 \mathbf s \cdot \mathbf z &= s_1 z_1 + \dots + s_m z_m \in \mathbb{C}^n. 
\end{align*}
\end{prop}
\begin{proof}
Using the relation $R(-\lambda, z) = -R(\lambda, -z)$, we can reduce the proof to the case where $\re(\lambda_j) > 0$ for $j=1, \ldots ,m$. From Lemma 
\ref{Resolvent_Neumann_series} we obtain that
\begin{multline*}
\Big{(}R(\lambda_1, z_1)^{k_1} \dots R(\lambda_m, z_m)^{k_m}\Big{)}^\sim(w) = \\
= \frac{i^{|\mathbf k| - m}}{(\mathbf k - \mathbf 1)!}\left \langle \frac{\partial^{\mathbf{k-1}}}{\partial \mu^{\mathbf{k-1}}}|_{\mu = \Lambda} 
R(\mu_1, z_1) \dots R(\mu_m, z_m)  k_w^{{\bf t}}, k_w^{\bf t}\right \rangle, 
\end{multline*}
where we use the short notation:
\begin{align*}
\frac{\partial^{\mathbf{k-1}}}{\partial \mu^{\mathbf{k-1}}}|_{\mu = \Lambda} := \frac{\partial^{k_1 - 1}}{\partial \mu_1^{k_1 - 1}} \dots \frac{\partial^{k_m - 1}}{\partial \mu_m^{k_m - 1}} |_{\mu_1 = \lambda_1, \dots, \mu_m = \lambda_m}.
\end{align*}
According to Lemma \ref{lemma:shiftofresolvent} we have: 
\begin{align*}
\Big{(} R(\mu_1, z_1)\ldots \:& R(\mu_m, z_m) \Big{)}^{\sim} (w)
= \alpha_{-w} \Big{(} R(\mu_1, z_1) \ldots R(\mu_m, z_m) \Big{)}^{\sim}(0)\\
&=\Big{(} \alpha_{-w}\big{[}R(\mu_1, z_1)\big{]}\ldots \alpha_{-w} \big{[} R(\mu_m, z_m) \big{]} \Big{)}^{\sim}(0)\\
&= \Big{(} R\big{(} \mu_1- 2i \sigma_{\mathbf t}(z_1,w), z_1 \big{)}  \ldots R\big{(} \mu_m-2i \sigma_{\mathbf t}(z_m,w), z_m\big{)} \Big{)}^{\sim}(0). 
\end{align*}
By using analyticity of  the resolvent maps $\mu_j \mapsto R(\mu_j, z_j) \in \mathcal L(F_{\mathbf t}^2)$ it follows that the difference quotients converge in operator norm. Hence, differentiation can be interchanged with the inner product, which yields:
\begin{align*}
&\frac{(\mathbf{k-1})!}{i^{|\mathbf k| - m}} \Big{(}R(\lambda_1, z_1)^{k_1}\dots R(\lambda_m, z_m)^{k_m}\Big{)}^\sim(w)\\
&\quad = \frac{\partial^{\mathbf{k-1}}}{\partial \mu^{\mathbf{k-1}}}|_{\mu = \Lambda}\Big{(} R(\mu_1, z_1)\ldots  R(\mu_m, z_m) \Big{)}^{\sim} (w)\\
&\quad = \frac{\partial^{\mathbf{k-1}}}{\partial \mu^{\mathbf{k-1}}}|_{\mu = \Lambda}
\underbrace{\Big{(} R\big{(} \mu_1- 2i \sigma_{\mathbf t}(z_1,w), z_1 \big{)}  \ldots R\big{(} \mu_m-2i \sigma_{\mathbf t}(z_m,w)\big{)} \Big{)}^{\sim}(0)}_{=(*)}.
\end{align*}
Now, we insert the integral expression of the resolvent in (\ref{resolventrepr3}): 
\begin{align*}
(*)&=
\int_{(0,\infty)^m}e^{-\sum_j [\mu_j+2i\sigma_{\mathbf t}(w, z_j))] s_j  - i\sum_{j < \ell} s_j s_\ell \sigma_{\mathbf t}(z_j, z_\ell)} \langle W_{{\bf s \cdot  z} }^{\mathbf t} 1, 1\rangle ~d\mathbf s. 
\end{align*}
The inner product in the integrand can be calculated explicitly:
\begin{equation*}
\langle W_{{\bf s \cdot  z} }^{\mathbf t} 1, 1\rangle =e	^{- \frac{1}{2} \|{\bf s \cdot z}\|_{\mathbf t}^2}. 
\end{equation*}
Finally, the assertion follows by inserting this expression into the last integral and performing the $\mu$-derivatives.  
\end{proof}
Similarly, one computes the Berezin transform of the classical resolvent functions:
\begin{lem}\label{Lemma_Berezin_transform_of_the_resolvent}
Let $m \in \mathbb{N}$, $z_1, \ldots, z_m \in \mathbb{C}^n$ and  ${\bf k}:=(k_1, \dots, k_m) \in {\mathbb N}^m$. For any set of complex numbers 
$\lambda_1, \ldots, \lambda_m \in \mathbb{C} \setminus i\mathbb{R}$ consider the function 
$$g := \big{(}\lambda_1 - 2i\sigma_{\mathbf t}(\cdot, z_1)\big{)}^{-k_1} \dots \big{(}\lambda_m - 2i\sigma_{\mathbf t}(\cdot, z_m)\big{)}^{-k_m}.$$ 
With the notation in Proposition \ref{formula:resolvent} the Berezin transform of $g$ is given by: 
\begin{align*}
\widetilde{g}^{(\mathbf t)}(w) :&= (T_g^{\mathbf t})^\sim(w) \\
&=\frac{1}{(\mathbf{k-1})!} \int_{(0, \infty)^m} \mathbf s^{\mathbf{k-1}} e^{- \Lambda \cdot \mathbf s + 2i\sigma_{\mathbf t}(w, \mathbf s \cdot \mathbf z) - \|\mathbf s \cdot \mathbf z\|_{\mathbf t}^2}~d\mathbf s.
\end{align*}
\end{lem}
\begin{proof}
The lemma follows by a direct calculation. First, note that
\begin{equation}\label{Integral_form_Berezin_transform_product_of_resolvents}
\widetilde{g}^{(\mathbf t)}(w) = \big{\langle} \alpha_{-w}(g) 1, 1\big{\rangle}\\
=\int_{\mathbb C^n} g(v+w)d\mu_{\mathbf t}(v). 
\end{equation}
Without loss of generality may assume that $\textup{Re}(\lambda_j)> 0$ for $j=1, \ldots, m$ such that $g$ has an integral representation (Laplace transform): 
\begin{equation*}
g(w)=\frac{1}{(\mathbf{k-1})!}\int_{(0,\infty)^m} \mathbf s^{\mathbf{k-1}} e^{-\Lambda \cdot \mathbf s + 2i\sigma_{\mathbf t}(w, \mathbf s \cdot \mathbf z)} d\mathbf s. 
\end{equation*}
Inserting the last expression into (\ref{Integral_form_Berezin_transform_product_of_resolvents}) and interchanging the
 order of integrations shows: 
\begin{multline*}
\widetilde{g}^{(\mathbf t)}(w) =\frac{1}{(\mathbf{k-1})!} \int_{(0, \infty)^m} \mathbf{s}^{\mathbf{k-1}} e^{-\Lambda \cdot \mathbf s + 2i\sigma_{\mathbf t}(w, \mathbf s \cdot \mathbf z)} 
\int_{\mathbb C^n} e^{2i\sigma_{\mathbf t}(v, \mathbf s \cdot \mathbf z)}~d\mu_{\mathbf t}(v)~d\mathbf{s}.
\end{multline*}
The inner integration can be evaluated explicitly:
\begin{equation*}
 \int_{\mathbb C^n} e^{2i\sigma_{\mathbf t}(v, \mathbf s \cdot \mathbf z)}~d\mu_{\mathbf t}(v)= \big{\langle} K_{\mathbf s \cdot \mathbf z}^{\mathbf t}, 
 K_{-\mathbf s \cdot \mathbf z}^{\mathbf t}\big{\rangle}
 =e^{-\| {\bf s \cdot  z}\|^2_{\mathbf t}} ,
\end{equation*}
which implies the statement of the lemma. 
\end{proof}
We now prove a first inclusion of algebras:
\begin{lem}\label{Proposition_inclusion_resolvent_algebra_Toeplitz_algebra}
It holds $R(\lambda, z) \in \mathcal T^{\mathbf t}(\mathcal R_{cl})$ for every $\lambda \in \mathbb C \setminus i \mathbb R$ and $z \in \mathbb C^n$. In particular, 
 \begin{equation*}
\mathcal{R}(\mathbb{C}^n, \sigma_{\bf t}) \subset  \mathcal T^{\mathbf t}(\mathcal R_{cl}).
 \end{equation*}
\end{lem}
\begin{proof}
Since we already know that $\mathcal R(\mathbb C^n, \sigma) \subset \mathcal C_1^{\mathbf t}$, it suffices to show that $\widetilde{R(\lambda, z)} \in \mathcal R_{cl}$ by the 
 correspondence theorem. Using $R(\lambda, z) = -R(-\lambda, -z)$ we may assume without loss of generality that $\re(\lambda) > 0$.
By transformation of the integral and according to Lemma \ref{Lemma_Berezin_transform_of_the_resolvent}: 
\begin{align*}
\widetilde{R(\lambda, z)}(w) &=-i \int_0^\infty e^{-\lambda s - 2is \sigma_{\mathbf t}(w, z) - \frac{s^2}{2} \| z\|_{\mathbf t}^2}~ds\\
&=i \sqrt{2} \int_0^{\infty} e^{- \lambda \sqrt{2} s- 2is \sigma_{\mathbf t}(\sqrt{2}w,z) -s^2\|z\|_{\mathbf t}^2} ds\\ 
&= i \sqrt{2}\:  \widetilde{g}^{(\mathbf t)}\big{(}\sqrt{2}w\big{)}, 
\end{align*}
where $g$ is a classical resolvent, namely: 
\begin{equation*}
g(w)=\big{(} \sqrt{2}\lambda -2i\sigma_{\mathbf t}(w,-z) \big{)}^{-1}. 
\end{equation*}
Note that the Berezin transform $\sim^{(\mathbf t)}$ behaves under dilations as follows: 
\begin{equation*}
\overset{\sim}{f}^{(\mathbf t)}(\rho w)= \rho^2 \cdot \overset{\sim}{f_{\rho}}^{{(\mathbf t/\rho^2)}}(w), \hspace{4ex} \rho >0, 
\end{equation*} 
where $f \in L^{\infty}(\mathbb{C}^n)$, $f_{\rho}(w):= f(\rho w)$ and $\mathbf t/\rho^2 = (t_1/\rho^2, \dots, t_n/\rho^2)$. Hence, we obtain: 
\begin{align*}
\widetilde{R(\lambda, z)}(w)=i 2 \sqrt{2} \cdot  \widetilde{g_{\sqrt{2}}}^{(\mathbf t/2)}(w). 
\end{align*}
Since $g_{\sqrt{2}}$ again is a resolvent function and $\mathcal{R}_{cl}$ is invariant under the Berezin transform  $\sim^{(\mathbf t/2)}$ (which is simply the convolution by an appropriate Gaussian function), the inclusion 
$\widetilde{R(\lambda, z)} \in \mathcal R_{cl}$ follows.
\end{proof}

Recall that an {\it isotropic subspace} $V \subset \mathbb C^n$ is a (real) subspace such that $\sigma_{\mathbf t}(z,w) = 0$ for all $z, w \in V$. Every isotropic subspace $V$ 
is of real dimension $\leq n$, and if $\dim_{\mathbb R}(V)= n$, then $V$ is called {\it Lagrangian}. To every Lagrangian subspace $V \subset \mathbb C^n$ 
there exists a complementary Lagrangian subspace $V' \subset \mathbb C^n$, i.e. $\mathbb C^n = V \oplus V'$. Indeed, one can choose $V' := \{ iz: z \in V\}$ and we will make this choice in the following for convenience. 

If we now fix a Lagrangian subspace $V$, then Proposition 
\ref{formula:resolvent} shows that the unital $C^\ast$ algebra
\begin{align*}
    \mathcal R_V := C^\ast\big{(}R(\lambda, z): ~\lambda \in \mathbb C \setminus i \mathbb R, ~z \in V\big{)}
\end{align*}
is commutative. Note that $ \mathcal R_V $ is also $\alpha$-invariant according to Lemma \ref{lemma:shiftofresolvent}. 

It is our next aim to show that $\mathcal R_V$, in the sense of Theorem \ref{Correspondence_Theorem}, corresponds to the space
\begin{align*}
    \mathcal R_{cl, V} := C^\ast \big{(}(\lambda - 2i\sigma_{\mathbf t}(\cdot, z))^{-1}: ~z \in V, ~\lambda \in \mathbb C \setminus i \mathbb R\big{)},
\end{align*}
i.e.
\begin{align*}
    \mathcal R_{V} = \mathcal T_{lin}^{\mathbf t}(\mathcal R_{cl, V}).
\end{align*}
Before we approach this goal, note the following facts: 

First, by Theorem 3.13 of \cite{Fulsche2021}, we have $\mathcal T_{lin}^{\mathbf t}(\mathcal R_{cl, V}) = \mathcal T^{\mathbf t}(\mathcal R_{cl, V})$, i.e. it suffices to prove that $\mathcal R_V = \mathcal T^{\mathbf t}(\mathcal R_{cl,V})$. Secondly, since $\mathcal R_{cl, V}$ is also invariant under dilations $f\mapsto f( \lambda\: \cdot)$ 
where  $\lambda >0$, we obtain $\mathcal R_V \subseteq \mathcal T^{\mathbf t}(\mathcal R_{cl, V})$ as in the proof of Lemma \ref{Proposition_inclusion_resolvent_algebra_Toeplitz_algebra}. Therefore, we only need to prove that $T_g^{\mathbf t} \in \mathcal R_V$ where $g$ is a product of classical resolvent functions $(\lambda - 2i\sigma_{\mathbf t}(\cdot, z))^{-1}$ with $z \in V$. Since the operator algebra $\mathcal R_V$ is commutative, we have more techniques at hand for obtaining this goal. In particular, Gelfand theory turns out to be useful here. The first conceptual goal is therefore describing the Gelfand spectrum of $\mathcal R_V$.
\vspace{1mm}\par 
Using the explicit formulas for the Berezin transform, one observes that point evaluations of the Berezin transforms are multiplicative linear functionals on $\mathcal R_V$. 
Since $\widetilde{A}(v + v') = \widetilde{A}(v')$ for every $v \in V$ and $v' \in V'$ and $A \in \mathcal{R}_V$, there is no loss of generality in considering only point evaluations of the Berezin transform at $V'$. Since the Berezin transform is injective, we can expect the Gelfand spectrum of $\mathcal R_V$ to be a suitable compactification of $V'$. We will describe this compactification now and start by recalling a compactification of a real inner product space first constructed in \cite{vanNuland}\footnote{To be more precise, therein it was only described for $\mathbb R^n$. It is straightforward to generalize the procedure to any finite dimensional real inner product space.}. Therein the author described the maximal ideal space of $\mathcal R_{cl}$ as such 
a compactification of $\mathbb R^{2n}$. 

Let $X$ be a finite-dimensional real inner product space. We denote by $P_Y$ the orthogonal projection onto a given subspace $Y\subset X$. By $\grass(X)$ we denote the 
{\it affine Grassmannian} of $X$, i.e. the set of all affine subspaces of $X$. As a set, this can be written as
\begin{align*}
    \grass(X) = \big{\{} x + Y: ~Y \text{\it  a linear subspace of } X \text{ \it and } x \perp Y\big{\}}.
\end{align*}
The precise topology with which $\grass(X)$ is endowed can be found in \cite{vanNuland} and will not be described here. We will denote by $\gamma X$ the set $\grass(X)$ endowed with this particular topology\footnote{In \cite{vanNuland}, the compactification is denoted by $\Omega$. We chose to name it differently, as the symbol $\Omega$ is somewhat ambiguous in a symplectic context}. We collect some facts about $\gamma X$ in the following lemma:
\begin{lem}[\cite{vanNuland}]
Let $X$ be a finite-dimensional real inner product space.
\begin{enumerate}
    \item $\gamma X$ is a compact Hausdorff space.
    \item Together with the embedding $X \ni x \mapsto x + \{ 0\} \in \gamma X$, $\gamma X$ is a compactification of $X$.
    \item A net $(x_{\iota} + Y_{\iota})_{\iota \in I} \subset \gamma X$ converges to $x+Y \in \gamma X$ if and only if the following hold:
    \begin{itemize}
        \item $P_{Y^\perp} x_{\iota} \overset{\iota \in I}{\longrightarrow} x$
        \item eventually $Y_{\iota} \subseteq Y$
        \item There is no affine subspace $x' + Y' \subsetneq x + Y$ such that there exists a subnet of $(x_{\iota} + Y_{\iota})$ with $P_{(Y')^\perp}x_{\iota} \to x'$ and $Y_{\iota} \subseteq Y'$ eventually (along the subnet). 
    \end{itemize}
\end{enumerate}
\end{lem}
One of the main theorems of \cite{vanNuland} is the following:
\begin{thm}[\cite{vanNuland}]
The Gelfand spectrum $\mathcal M(\mathcal R_{cl})$  of $\mathcal R_{cl}$ can be identified with the above compactification of $\mathbb R^{2n}$, i.e. $\mathcal M(\mathcal R_{cl}) \cong \gamma \mathbb R^{2n}$.
\end{thm}
Let us briefly describe the identification of $\mathcal M(\mathcal R_{cl})$ with $\gamma \mathbb R^{2n}$ in more detail: 
Given a function $f \in \mathcal R_{cl}$ and an affine line $x + \Span\{ y\} \subset \mathbb R^{2n}$, $y \in \mathbb R^{2n}$, it is not hard to verify 
that $\lim_{\alpha \to \infty} f(x + \alpha y)$ exists. For an affine subspace $x + Y \subset \mathbb R^{2n}$ and any $y \in Y$ with $\| y\| = 1$, 
the value of this limit is almost everywhere, with respect to the surface measure on $\{ y \in Y: ~\| y \| = 1\}$, the same. If we denote this value by $\varphi_{x + Y}(f)$, then this defines a multiplicative linear functional and every element of $\mathcal M(\mathcal R_{cl})$ can be obtained in this way.

Indeed, the same can be done for $\mathcal R_{cl, V}$, and the following holds true:
\begin{prop}
    The Gelfand spectrum $\mathcal M(\mathcal R_{cl, V})$ of $\mathcal R_{cl, V}$ can be identified with the above compactification of
     $V^{\prime}$, i.e. $\mathcal M(\mathcal R_{cl, V}) \cong \gamma V'$.
\end{prop}
Here, we consider $V'$ as a real inner product space with the $\mathbf t$-weighted inner product induced from $\mathbb R^{2n} \cong \mathbb C^n$, i.e. $\langle z, w\rangle = \sum_{j=1}^n \frac{\re(z \cdot \overline w)}{t_j}$. We will not need the result in its fulll strength and we leave it as an exercise to adapt the arguments from \cite{vanNuland}. It is sufficient and elementary to verify that via the embedding $V' \ni v \mapsto v + \{ 0\} \in \gamma V'$, every classical resolvent $(\lambda - 2i\sigma_{\mathbf t}(\cdot, z))^{-1}$ with $z \in V$ extends to a function in $C(\gamma V')$.

As we will see next, the Gelfand spectrum of $\mathcal R_V$ is indeed the same as the one of $\mathcal R_{cl, V}$:
\begin{prop}\label{prop:specRv}
    The Gelfand spectrum $\mathcal M(\mathcal R_V)$ of $\mathcal R_V$ can be identified with $\mathcal M(\mathcal R_V) \cong \mathcal M(\mathcal R_{cl, V}) \cong \gamma V'$, where all multiplicative linear functionals on $\mathcal R_V$ have the form: 
    \begin{align*}
        \psi_{x + Y}(A) = \varphi_{x + Y}\big{(}\widetilde{A}|_{V'}\big{)}. 
    \end{align*}
    Here, we denote by $\varphi_{x+Y}$ the multiplicative functional of $\mathcal R_{cl, V}$ introduced above. 
\end{prop}
Before proving Proposition \ref{prop:specRv}, let us derive our intended result from this:
\begin{thm}\label{Main_theorem_R_V_T_lin}
It is $\mathcal R_V = \mathcal T_{lin}^{\mathbf t}(\mathcal R_{cl, V})$.
\end{thm}
\begin{proof} The inclusion $\mathcal R_V \subset \mathcal T_{lin}^{\mathbf t}(\mathcal R_{cl, V})$ was already stated above.
Let $g \in \mathcal R_{cl, V}$. Since $\mathcal R_{cl, V}$ is translation invariant and closed, we conclude that 
$\widetilde{g}^{(\mathbf t)} = \widetilde{T_g^{\mathbf t}} \in \mathcal R_{cl, V} = C(\mathcal M(\mathcal R_V))$. Therefore, $\widetilde{g}^{(\mathbf t)}$ is 
the Gelfand transform of an operator in $\mathcal R_V$. This operator must be $T_g^{\mathbf t}$, i.e. $T_g^{\mathbf t} \in \mathcal R_V$.
\end{proof}
\noindent
We will present a sequence of lemmas that lead to a proof of Proposition \ref{prop:specRv}.
\begin{lem}
For any resolvent $R(\lambda, z) \in \mathcal R_V$, where $z \in V$, and $x, y \in V'$ it holds: 
\begin{align*}
\widetilde{R(\lambda, z)}(x+y) = \widetilde{R(\lambda, z)}\big{(}x + P_{\Span \{ iz\}}y\big{)}.
\end{align*}
\end{lem}
\begin{proof}
Let $z \in V$ and first observe that $\sigma_{\mathbf t}(z, y) = 0$ iff $y \perp iz$ (note that $iz \in V'$ by our choice $V' = \{ iz: z \in V\}$). Hence, we have
\begin{align*}
   \widetilde{R(\lambda, z)}(x+y) &= \widetilde{R(\lambda, z)}\big{(}x + (I-P_{\Span\{ iz\}})y + P_{\Span\{iz\}} y\big{)} \\
    &= \widetilde{R(\lambda, z)}\big{(}x + P_{\Span\{ iz\}}y\big{)},
\end{align*}
where the last equality follows from the formula for the Berezin transform of a resolvent in Proposition \ref{formula:resolvent}.
\end{proof}
\begin{lem}\label{lem:limit_on_lines}
For any resolvent $R(\lambda, z) \in \mathcal R_V$ where $z \in V$ and any affine line $x + \Span\{ y\} \in \gamma V'$ the 
limit $\lim_{\alpha \to \infty} \widetilde{R(\lambda, z)}(x + \alpha y)$ exists and is given by
\begin{align}\label{value_of_the_BT_on_AL}
    \lim_{\alpha \to \infty} \widetilde{R(\lambda, z)}(x + \alpha y) = \begin{cases}
    \widetilde{R(\lambda, z)}(x), \quad &\sigma_{\mathbf t}(z, y) = 0,\\
    0, \quad &\sigma_{\mathbf t}(z, y) \neq 0.
    \end{cases}
\end{align}
\end{lem}
\begin{proof}
If $\sigma_{\mathbf t}(z, y) = 0$, then $y \perp iz$ such that $P_{\Span \{ iz\}}y=0$ and the equality follows from the previous lemma. Assume that $\sigma_{\mathbf t}(z, y) \neq 0$ and let 
$\re(\lambda) > 0$  (the case $\re(\lambda) > 0$ follows similarly). According to Proposition \ref{formula:resolvent} the Berezin transform of the resolvent at $x+\alpha y$ has the value: 
\begin{align*}
    \widetilde{R(\lambda, z)}(x + \alpha y) = \int_0^\infty e^{-\lambda s - 2is\alpha \sigma_{\mathbf t}( y, z) - 2is\sigma_{\mathbf t}(x, z) - \frac{s^2}{2}\| z\|_{\mathbf t}^2}~ds.
\end{align*}
We use integration by parts 
in order  to show that the right hand side converges to $0$ as $\alpha \rightarrow \infty$. With the short notation $\sigma := \sigma_{\mathbf t}(y, z) \neq 0$ we have: 
\begin{align*}
    &\widetilde{R(\lambda, z)}(x + \alpha y)\\
    &\ = -\frac{1}{2i\alpha \sigma} \int_0^\infty \frac{\rm d}{{\rm d}s}\left [ e^{-2is\alpha \sigma} \right] e^{-\lambda s - 2is\sigma_{\mathbf t}(x, z) - \frac{s^2}{2}\| z\|_{\mathbf t}^2}~ds\\
    &\ = -\frac{1}{2i\alpha \sigma} \Big( \underbrace{\left [ e^{-2is\alpha\sigma - \lambda s - 2is\sigma_{\mathbf t}(x, z) - \frac{s^2}{2}\| z\|_{\mathbf t}^2}\right ]_{s=0}^\infty}_{=-1}\\
    &\hspace{1cm} - \int_0^\infty e^{-2is\alpha \sigma} \big{(}-\lambda - 2i\sigma_{\mathbf t}(x, z) - s\| z\|_{\mathbf t}^2\big{)} e^{-\lambda s - 
    2is\sigma_{\mathbf t}(x, z) - \frac{s^2}{2}\| z\|_{\mathbf t}^2}~ds\Big).
\end{align*}
Since the integral 
\begin{align*}
    \int_0^\infty \big{|}\lambda + 2i\sigma_{\mathbf t}(x, z) + s\| z\|_{\mathbf t}^2\big{|} e^{-\lambda s - \frac{s^2}{2}\| z\|_{\mathbf t}^2}~ds
\end{align*}
over the absolute value is finite, we obtain $\widetilde{R(\lambda, z)}(x + \alpha y) \to 0$ as $\alpha \to \infty$.
\end{proof}
\begin{lem}
Let $R(\lambda, z)$ with $z \in V$ be a resolvent in $\mathcal R_V$. Then, its restriction $\widetilde{R(\lambda, z)}|_{V'}$ to $V^{\prime}$ extends to a continuous 
function on $\gamma V'$.
\end{lem}
\begin{proof}
We first describe the values that (the extension of) $\widetilde{R(\lambda, z)}|_{V'}$ takes on $\gamma V'$ and discuss continuity in a second step. 

The value of the Berezin transform at $x \cong x + \{0\}$ is given by $\widetilde{R(\lambda, z)}(x)$. At each affine line $x + \Span\{ y\}$ the value of the extension of the Berezin transform to $\gamma V'$ is defined by the right hand side of (\ref{value_of_the_BT_on_AL}). For a given affine spaces $x + Y$ with $\dim_{\mathbb R}(Y) \geq 2$, we distinguish 
two cases: If $Y \subseteq \{ w \in V': \sigma_{\mathbf t}(w, z) = 0\}$, then the value of $\widetilde{R(\lambda, z)}(x+Y)$ coincides with $\widetilde{R(\lambda, z)}(x)$. Otherwise, 
if $Y \not \subseteq \{ w \in V': \sigma_{\mathbf t}(w, z) = 0\}$, then we set it to be zero. 
\vspace{1mm}\par 
Note that these choices are in accordance with the definition of the multiplicative linear functional $\varphi_{x+Y}$. We explain what we mean by this only in the last example: 
If $Y \not \subseteq \{ w\in V'\: : \:  \sigma_{\mathbf t}(w, z) = 0\}$, then $\{ w \in Y: \sigma_{\mathbf t}(w, z) = 0\}$ is a subspace of $Y$ of dimension strictly smaller than the dimension of $Y$. Hence, the unit sphere of $\{ w \in Y: \sigma_{\mathbf t}(w, z) = 0\}$ is a zero set with respect to the surface measure of the unit sphere of $Y$. Hence, 
$\varphi_{x + Y}(\widetilde{R(\lambda, z)}|_{V^{\prime}}) = 0.$

It remains to prove that the above extension of $\widetilde{R(\lambda, z)}|_{{V^{\prime}}}$ from $V^{\prime}$ to the compactification $\gamma V'$ is in fact continuous. According to 
{\it Bourbaki's Extension Theorem}, \cite[Theorem 1, p. 82]{Bourbaki},  it suffices to show that 
\begin{equation}\label{Continuity_radially}
\widetilde{R(\lambda, z)}(x_{\iota}) \to \varphi_{x+Y}\big{(}\widetilde{R(\lambda, z)}\big{)}
\end{equation}
for any net $(x_{\iota})_{\iota \in I} \subset V'$  such that $x_{\iota} + \{ 0\} \to x + Y$. We distinguish several cases:
\begin{enumerate}
    \item If $x_{\iota} + \{ 0\} \to x + \{ 0\}$, then (\ref{Continuity_radially}) follows from the continuity of the Berezin transform $\widetilde{R(\lambda,z)}$ on $V^{\prime}$.
    \item If $x_{\iota} + \{ 0\} \to x + Y$ with $Y \subseteq \{ w \in V': ~\sigma_{\mathbf t}(w, z) = 0\}$, then:
    \begin{align*} 
    \widetilde{R(\lambda, z)}(x_{\iota}) &= \widetilde{R(\lambda, z)}\big{(}x_{\iota} - P_{Y}x_{\iota}\big{)} = \widetilde{R\big{(}\lambda, z)}\big{(}P_{Y^\perp}x_{\iota}\big{)} \\
    &\overset{\iota}{\longrightarrow} \widetilde{R(\lambda, z)}(x) = \varphi_{x + Y}\big{(}\widetilde{R(\lambda, z)}\big{)}.
    \end{align*}
    \item If $x_{\iota} + \{ 0\} \to x + Y$ with $Y \not \subseteq \{ w\in V': ~\sigma_{\mathbf t}(w, z) = 0\}$, then we conclude that 
    $$\widetilde{R(\lambda, z)}(x_{\iota}) \overset{\iota}{\longrightarrow} 0 = \varphi_{x + Y}\big{(}\widetilde{R(\lambda, z)}\big{)}.$$ 
    In fact, assuming the opposite there exists a subnet, also denoted by $(x_{\iota})$, such that $\widetilde{R(\lambda, z)}(x_{\iota})$ is bounded away from zero. 
    Take $w_0 \in Y$ such that $\sigma_{\mathbf t}(w_0, z) \ne 0$ and put $Z:= \textup{span}\{w_0\}\subset Y$. By $Y^{\prime}\subset Y$ denote the orthogonal complement 
    of $Z$ in $Y$. We obtain a decomposition of $V^{\prime}$: 
     \begin{equation}\label{orthogonal_decomposition_V_prime}
    V^{\prime} = Y ^{\perp} \oplus Y=Y^\perp \oplus Y' \oplus Z. 
    \end{equation}
    \begin{enumerate} 
    \item[(a)] If $P_Z x_{\iota}$ was unbounded, we could pass to a subnet such that $\| P_Z x_{\iota}\| \to \infty$. For this subnet, one could show by using the method of stationary phase 
    (as in the proof of Lemma \ref{lem:limit_on_lines}) that $\widetilde{R(\lambda, z)}(x_{\iota}) \to 0$. However, this is impossible as we assumed that 
    $\widetilde{R(\lambda, z)}(x_{\iota})$ is bounded away from zero.
    \item[(b)] If $P_Z x_{\iota}$ is bounded, we can pass to a subnet such that $P_Z x_{\iota}$ converges (say, to some $x_0 \in Z$). 
   From the orthogonal decomposition (\ref{orthogonal_decomposition_V_prime}) we see that:
    \begin{align*}
        P_{(Y')^\perp} x_{\iota} = P_{Y^\perp} x_{\iota} + P_Z x_{\iota} \overset{\iota}{\longrightarrow} x + x_0\in Y^{\perp}\oplus Z.
    \end{align*}
   From $x_0 \in Z \subset Y$ it follows: 
       $x + x_0 + Y' \subset x + Y$.
    However this contradicts the assumption $x_{\iota} + \{ 0\} \overset{\iota}{\longrightarrow} x + Y$. \qedhere
    \end{enumerate}
\end{enumerate}
\end{proof}
Since the Berezin transform is multiplicative on $V'$ and $A_k \to A$ in $\mathcal R_V$ implies $\widetilde{A_k} \to \widetilde{A}$ uniformly, we conclude that $\widetilde{A}$ 
for every $A \in \mathcal R_V$ extends to a continuous function on $\gamma V'$. Now we can give the proof of   Proposition \ref{prop:specRv}: 
\begin{proof} [Proof of Proposition \ref{prop:specRv}]
We claim that the map 
$$\Phi: \mathcal R_V \to C(\gamma V'): \: \Phi(A) = \big{[}(x + Y) \mapsto  \varphi_{x+Y}\big{(}\widetilde{A}_{|_{V^{\prime}}}\big{)}\big{]}$$
is a bijective and unital homomorphism of $C^\ast$ algebras. By what has been said before, $\Phi$ defines a continuous and injective unital $\ast$-homomorphism. 
We are left with verifying surjectivity. Since the range of $\Phi$ is a $\ast$-algebra, we only need to show that it separates points of $\gamma V'$ in order to apply the 
Stone-Weierstrass Theorem. Showing that the points are separated by $\Phi(\mathcal R_V)$ is easy, as we can always choose the Berezin transform of a resolvent 
for separating two given sets. Careful inspection of the proof of the previous lemma indeed gives a guideline on how to choose the resolvent. We briefly explain this here. 
In what follows, we always assume that $\re(\lambda) > 0$.
\begin{enumerate}
    \item $0 + V'$ is separated from any other affine subspace in the following way: If $Y \subsetneq V'$ is a proper subspace, then let $0 \neq v \perp Y$ and consider 
    $R(\lambda, i\alpha v)$ for $\alpha \in \mathbb R$. Then, $\alpha$ can be chosen such that 
  \begin{equation*}
   \Phi\big{(}R(\lambda, i\alpha v)\big{)}(v+Y)=\psi_{v+Y}\Big{(}R(\lambda, i\alpha v) \Big{)}=\widetilde{R(\lambda, i\alpha v)}(v) \neq 0.
    \end{equation*}
     However, $\psi_{0 + V'}(R(\lambda, i\alpha v)) = 0$ as long as $\alpha \neq 0$. Hence the affine spaces $v + Y$ and $0 + V'$ are separated. Separating $0 + V'$ from 
     $0 + Y$ can be done by considering $R(\lambda, i v)$.
    \item Affine spaces $x_1 + Y \ne x_2 + Y$ are separated as follows: We necessarily have $Y \not = V'$ and $x_1 \ne x_2$. If $x_1 = \rho x_2$ for some $\rho \in \mathbb R$, then consider $R(\lambda, i\alpha x_2)$ for suitable $\alpha \in \mathbb R$.
    Otherwise, we can find $z_0 \perp Y$ such that $P_{\Span \{ z_0\}} x_1 \neq P_{\Span \{ z_0\}} x_2$. Let $z = iz_0 \in V$ such that $Y \subset \{ w \in V': \sigma_{\mathbf t}(w, z) = 0\}$. Hence for $\alpha \in \mathbb R \setminus \{ 0\}$:
    \begin{align*}
        \psi_{x_1 +Y}\big{(}R(\lambda, \alpha z)\big{)} &= \widetilde{R(\lambda, \alpha z)}\big{(}P_{\Span \{ z_0\}} x_1\big{)},\\
        \psi_{x_2 + Y}\big{(}R(\lambda, \alpha z)\big{)} &= \widetilde{R(\lambda, \alpha z)}\big{(}P_{\Span \{ z_0\}} x_2\big{)}.
    \end{align*}
    Now, one has to choose $\alpha$ accordingly such that these values are different (see the formula in Proposition \ref{formula:resolvent}). 
    \item If $Y_1 \neq Y_2$ are proper subspaces of $V'$, then we consider two cases: 
    \begin{enumerate}
        \item[(a)] Recall that we always have $x \perp Y$ for an affine subspace $x+Y$. If $x_1= \rho x_2$ with $\rho \in \mathbb{R}$ are real multiples of each other, then we obtain with $\alpha \in \mathbb R$:
        \begin{align*}
            \psi_{x_1 + Y_1}\big{(}R(\lambda, i\alpha x_2)\big{)} &= \widetilde{R(\lambda, i\alpha x_2)}(\rho x_2) \\
            &= \int_0^\infty e^{-\lambda s}e^{-2is\rho \alpha \| x_2\|_{\mathbf t}^2 - \frac{s^2 \alpha^2}{2} \| x_2\|_{\mathbf t}^2}~ds,\\
            \psi_{x_2 + Y_2}\big{(}R(\lambda, i\alpha x_2)\big{)} &= \widetilde{R(\lambda, i\alpha x_2)}(x_2) \\
            &= \int_0^\infty e^{-\lambda s}e^{-2is\alpha \| x_2\|_{\mathbf t}^2 - \frac{s^2 \alpha^2}{2} \| x_2\|_{\mathbf t}^2}~ds.
        \end{align*}
        Now, arrange $\alpha$ such that these are not equal.
        \item[(b)] If $x_1$ and $x_2$ are not real multiples of each other, then it must be either $P_{\Span \{ x_1\}} x_2 \neq x_1$ or $P_{\Span \{ x_2\}} x_1 \neq x_2$. Assume that
        $P_{\Span \{ x_1\}}x_2 \neq x_1$. For $\alpha \in \mathbb R$:
        \begin{align*}
            \psi_{x_1 + Y_1}\big{(}R(\lambda, i\alpha x_1)\big{)} &= \widetilde{R(\lambda, i\alpha x_1)}(x_1) \\
            &= \int_0^\infty e^{-\lambda s}e^{-2is\alpha \| x_1\|_{\mathbf t}^2 - \frac{s^2 \alpha^2}{2} \| x_1\|_{\mathbf t}^2}~ds.
        \end{align*}
        The value at $x_2 + Y_2$ now depends on whether $x_1$ is orthogonal to $Y_2$ or not: If $x_1 \perp Y_2$, then
        \begin{align*}
            \psi_{x_2 + Y_2}\big{(}R(\lambda, i\alpha x_1)\big{)} &= \widetilde{R(\lambda, i\alpha x_1)}(x_2) \\
            &= \int_0^\infty e^{-\lambda s} e^{-2is\alpha \sigma_{\mathbf t}(P_{\Span \{ ix_1\}} x_2, x_1) -  \frac{s^2\alpha^2}{2} \| x_1\|_{\mathbf t}^2}~ds.
        \end{align*}
        Since $\sigma_{\mathbf t}(P_{\Span \{ ix_1\}} x_2, x_1) \neq \| x_1\|_{\mathbf t}^2$, we can arrange $\alpha$ such that the values of 
        $\psi_{x_j + Y_j}(R(\lambda, i\alpha x_1))$ for $j=1,2$ are different.
         If, on the other hand, $x_1 \not \perp Y_2$, then
        \begin{align*} 
            \psi_{x_2 + Y_2}\big{(}R(\lambda, i\alpha x_1)\big{)} = 0.
        \end{align*}
        Hence, we have to arrange $\alpha$ such that $\psi_{x_1 + Y_1}(R(\lambda, i\alpha x_1)) \neq 0$ (e.g. by letting $\alpha = 0$).\qedhere
    \end{enumerate}
\end{enumerate}
\end{proof}
Before continuing, we state a consequence of the previous considerations: 
\begin{cor}\label{lag_subspace:1}
Let  $m \in \mathbb{N}$ and $\lambda_1, \ldots , \lambda_m\in \mathbb{C} \setminus i \mathbb{R}$. Consider the function 
$$g = \big{(}\lambda_1 - 2i\sigma_{\mathbf t}(\cdot, z_1)\big{)}^{-k_1} \cdot \dots \cdot \big{(}\lambda_m - 2i\sigma_{\mathbf t}(\cdot, z_m)\big{)}^{-k_m},$$ 
where $z_1, \dots, z_m\in V$ and $V \subset \mathbb{C}^n$ is Lagrangian. Then, $T_g^{\mathbf t} \in \mathcal R(\mathbb C^n, \sigma_{\mathbf t})$.
\end{cor}
\begin{proof}
By definition we have $g \in \mathcal R_{cl, V}$. Now Theorem \ref{Main_theorem_R_V_T_lin} implies that 
 $T_g^{\mathbf t} \in \mathcal T_{lin}^{\mathbf t}(\mathcal R_{cl, V}) = \mathcal R_V \subset \mathcal R(\mathbb C^n, \sigma_{\mathbf t})$.
\end{proof}
Now we consider resolvent functions for which the vectors $z_j$ are either taken from  $V$ or from $V^{\prime}$. 
{\begin{cor}\label{lag_subspace:2}
Let $V$ and $V'$ be two Lagrangian subspaces of $\mathbb C^n$ such that $\mathbb C^n = V \oplus V'$. Let $m, \ell \geq n$ and $z_1, \dots, z_m \in V$ and $w_1, \dots, w_\ell \in V'$ such that $V = \Span_{\mathbb R}\{ z_1, \dots, z_m\}$ and $V' = \Span_{\mathbb R} \{ w_1, \dots ,w_\ell\}$. Set
\begin{align*}
    g = \prod_{j=1}^{m} \big{(}\lambda_j - 2i\sigma_{\mathbf t}(\cdot, z_j)\big{)}^{-k_j} \cdot \prod_{j=1}^{\ell}\big{(}\lambda_j' - 2i\sigma_{\mathbf t}(\cdot, w_j)\big{)}^{-k_j'},
\end{align*}
where $k_j,k_j' \in \mathbb{N}$ and $\lambda_j, \lambda_j' \in \mathbb C \setminus i \mathbb R$. Then, $T_g^{\mathbf t} \in \mathcal R(\mathbb C^n, \sigma_{\mathbf t})$.
\end{cor}
\begin{proof}
For $v = v_1 + v_2 \in V \oplus V' = \mathbb C^n$ we have
\begin{align*}
    g(v) = \prod_{j=1}^m \big{(}\lambda_j - 2i\sigma_{\mathbf t}(v_2, z_j)\big{)}^{-k_j} \cdot \prod_{j=1}^\ell \big{(}\lambda_j' - 2i\sigma_{\mathbf t}(v_1, w_j)\big{)}^{-k_j'}.
\end{align*}
As $v_1 \to \infty$, the first factor tends to $0$. As $v_2 \to \infty$, the second factor tends to $0$. In conclusion, $g \in C_0(\mathbb C^n)$. Hence, $T_g^{\mathbf t} \in \mathcal K(F_{\mathbf t}^2)$. Since $\mathcal K(F_{\mathbf t}^2) \subset \mathcal R(\mathbb C^n, \sigma_{\mathbf t})$ according to Proposition \ref{Proposition_inclusion_of_spaces} the statement follows.
\end{proof}}
\begin{thm}\label{resolvent:dim1}
Let $n = 1$, then $\mathcal D_0^t = \mathcal R_{cl}$. 
\end{thm}
\begin{proof}
If $n=1$, then any classical resolvent function $g$ either fulfills the assumption of Corollary \ref{lag_subspace:1} or the assumptions of Corollary \ref{lag_subspace:2}. In either case, we obtain 
$T_g^{\mathbf t} \in \mathcal R(\mathbb C^n, \sigma_{\mathbf t})$.
\end{proof}

For $n > 1$, the situation is more complicated. We will show only a weaker result for this case.

Since every $z \in \mathbb C^n$ is contained in some Lagrangian subspace $V$, we conclude that each resolvent $R(\lambda, z)$ is contained in some 
$\mathcal R_V$. Thus, we see that 
\begin{align*}
   \mathcal R_0 :=  \sum_{V \text{ Lagrangian subspace}} \mathcal R_V
\end{align*}
contains every resolvent $R(\lambda, z)$. Further, since $\mathcal R_V = \mathcal T_{lin}^{\bf t}(\mathcal R_{cl, V})$, we obtain
\begin{align*}
    \overline{\mathcal R_0} = \mathcal T_{lin}^{\mathbf t}(\overline{\mathcal R_{cl, 0}}),
\end{align*}
where $\mathcal R_{cl, 0}$ is given by
\begin{align*}
    \mathcal R_{cl, 0} := \sum_{V \text{ Lagrangian subspace}} \mathcal R_{cl, V}.
\end{align*}
As a weak substitute for Theorem \ref{resolvent:dim1} in higher dimensions, we get:
\begin{prop}\label{Propositon_resolvent_algebra_is_TA}
    $\mathcal R(\mathbb C^n, \sigma_{\mathbf t}) = \mathcal T^{\mathbf t}(\overline{\mathcal R_{cl, 0}})$.
\end{prop}
It might of course be true that $\overline{\mathcal R_{cl, 0}} = \mathcal R_{cl}$. In this case, the previous result would  imply $\mathcal D_0^t = \mathcal R_{cl}$. Nevertheless, so far this remains an open question.
\section{Infinite dimensional symplectic space}
\label{Infinite dimensional symplectic space}
{As is well-known there are significant differences between the resolvent algebras of finite  and infinite dimensional symplectic spaces: in finite dimensions} 
every two regular irreducible representations are unitarily equivalent.  However, in the case of infinite dimensional symplectic spaces, this statement is false. Nevertheless, 
we can build particular representations of the resolvent algebra by using ideas from the previous sections. 
\vspace{1mm}\par 
Let $\mathbf t = (t_k)_{k=1}^\infty\in \ell^1(\mathbb N)$ be a summable sequence of strictly positive real numbers. Then, we set $\mathcal H:= \ell^2(\mathbb N)$ with the usual norm 
$\|\cdot\|_{\ell^2}$ and define: 
\begin{align*}
\mathcal H_{1/2}^{\mathbf t} &:= \Big{\{} z \in \ell^2(\mathbb N): ~\sum_{n=1}^\infty \frac{|z_n|^2}{t_n} < \infty \Big{\}},\\
\mathcal H_1^{\mathbf t} &:= \Big{\{} z \in \ell^2(\mathbb N): ~\sum_{n=1}^\infty \frac{|z_n|^2}{t_n^2} < \infty\Big{\}}.
\end{align*}
We may think of $\mathcal H_1^{\mathbf t}$ and $\mathcal H_{1/2}^{\mathbf t}$ as the range of $B$ and $\sqrt{B}$, respectively, where $B$ is the trace class 
operator obtained by linearly extending the map $e_n \mapsto t_n e_n$ with $\{ e_n\}$ being the standard orthonormal basis of $\ell^2(\mathbb N)$. 
We consider $\mathcal H_{1/2}^{\mathbf t}$ as a symplectic Hilbert space equipped with the symplectic form:
\begin{align*}
    \sigma_{\mathbf t}(z,w) := \sum_{n=1}^\infty \frac{\im(z_n \cdot \overline{w_n})}{t_n}, \hspace{3ex} z=(z_n), w=(w_n) \in \mathcal H_{1/2}^{\mathbf t}.
\end{align*}
Note that $\sigma_{\mathbf t}(z,w)$ is well-defined even for $z\in \mathcal H_1^{\mathbf t}$ and $w \in \mathcal H$. We can write $\sigma_{\mathbf t} = \im \langle \cdot, \cdot\rangle_{1/2}$, where $\langle \cdot, \cdot\rangle_{1/2}$ is the canonical inner product of $\mathcal H_{1/2}^{\mathbf t}$:
\begin{align*}
    \langle z, w\rangle_{1/2} := \sum_{n=1}^\infty \frac{z_n \cdot \overline{w_n}}{t_n}.
\end{align*}
Moreover, $\| \cdot\|_{1/2}$ denotes the corresponding norm: $\| z\|_{1/2}^2 = \langle z,z\rangle_{1/2}$.
\vspace{1mm}\par 
We recall some basic facts about Gaussian measures on infinite dimensional Hilbert spaces, cf. \cite{Skorohod,DaPrato, Bauer_Rodriguez}. The infinite product measure $\mu_{\bf t}:=\prod_{k=1}^\infty \mu_{t_k}$ of the Gaussian measures 
\begin{equation*}
d\mu_{t_k}(z):= \frac{1}{\pi t_k} e^{-\frac{|z|^2}{t_k}} dV(z), \hspace{4ex} z\in \mathbb{C}
\end{equation*}
gives a well-defined probability measure on $\mathbb C^\infty = \prod_{k=1}^\infty \mathbb C$. It is concentrated on $\mathcal H= \ell^2(\mathbb{N}) \subset \mathbb{C}^{\infty}$. 
By $\mu_{\mathbf t}$, we also denote its restriction to the measurable space $(\mathcal H, \mathcal B)$, where $\mathcal B$ is the Borel $\sigma$-algebra of $\mathcal H$. Note that $\mathcal B$ agrees with the $\sigma$-algebra generated by cylindrical Borel sets. The measure $\mu_{\mathbf t}$ is called the (centered) {\it Gaussian measure} on $\mathcal H$ 
with {\it covariance operator} $B$: This is due to the fact that $B$ is naturally related to $\mu_{\mathbf t}$ by
\begin{align}\label{Gaussmeasure_Fouriertrafo}
    \int_{\mathcal H} e^{i\re \langle x, y\rangle_{\ell^2}}~d\mu(x) = e^{-\frac 12 \langle By, y\rangle_{\ell^2}}, \hspace{4ex} (y \in \mathcal{H}).
\end{align}
The space $\mathcal H_{1/2}^{\mathbf t}$ introduced above is also known as the {\it Cameron-Martin space} of $\mu_{\mathbf t}$. It is a measurable set of measure zero: $\mu_{\mathbf t}(\mathcal H_{1/2}^{\mathbf t}) = 0$. 
Hence, so is $\mathcal H_1$. 
\vspace{1mm}\par 
We will now introduce the {\it Fock-Bargmann space} of holomorphic functions in infinitely many variables and Berezin-Toeplitz quantization in this setting. Two different approaches have been proposed in  \cite{Janas_Rudol1990, Janas_Rudol1995}: the first being of {\it measure theoretic} nature and the second based on an {\it inductive limit construction}. Both approaches yield different theories and we deal with the measure theoretic approach here. 

We will write $\mathcal Z_0$ for the set of all sequences $(\alpha_n)_{n=1}^\infty$ with values in $\mathbb N_0$ such that all but finitely many entries are zero.  In the following 
we use the standard multi-index notation: let $z=(z_1,z_2, \ldots) \in \ell^2(\mathbb{N})$ and $\alpha \in \mathcal{Z}_0$, then we put: 
\begin{align*}
    z^\alpha = z_1^{\alpha_1}z_2^{\alpha_2}\dots, \quad \alpha! = \alpha_1! \alpha_2! \dots, \quad \mathbf t^\alpha = t_1^{\alpha_1} t_2^{\alpha_2} \dots.
\end{align*}
The monomials  $\mathcal{M}$ below form an orthonormal system of complex analytic functions on $\mathcal{H}$ inside $L^2(\mathcal{H}, \mu_{\bf t})$: 
\begin{equation*}
\mathcal{M}:=\Big{[}e_\alpha(z) := \frac{1}{\sqrt{\mathbf t^{\alpha} \alpha!}} z^\alpha: ~\alpha \in \mathcal Z_0\Big{]}. 
\end{equation*}
We write $\langle \cdot, \cdot \rangle$ and $\|\cdot \|$ for the standard inner product and norm in $L^2(\mathcal{H}, \mu_{\bf t})$, respectively. 
Elements in the linear span of $\mathcal{M}$ are referred to as {\it analytic polynomials}. In analogy to the finite dimensional setting of $\mathbb{C}^n$ we define the Bargmann-Fock space $F_{\mathbf t}^2(\mathcal H)$ to be the 
$L^2$-closure of analytic polynomials: 
\begin{equation*}
    F_{\mathbf t}^2(\mathcal H) : = \overline{\operatorname{span }}\:  \mathcal{M}.
\end{equation*}

Let $z \in \mathcal H_1$ and $p \in \textup{span}\: \mathcal{M}$ be an analytic polynomial. Extending (\ref{definition_Weyl_operator}) we define the {\it Weyl operator} $W_z^{\mathbf t}$ on $F_{\mathbf t}^2(\mathcal H)$ by:
\begin{align*}
    W_z^{\mathbf t}p(w) := k_z^{\mathbf t}(w)p(w-z), \hspace{4ex} w \in \mathcal{H}. 
\end{align*}
Here, $k_z^{\mathbf t}$ denotes the {\it normalized reproducing kernel} defined by: 
\begin{align*}
    k_z^{\mathbf t}(w) = e^{\langle w, z\rangle_{1/2} - \frac{1}{2}\| z\|_{1/2}^2}.
\end{align*}
Let $N \in \mathbb{N}$ and assume that the analytic polynomial $p$ only depends on the variables $z_1, \ldots, z_N$. We write 
$\mu_{\mathbf t} = \mu_{\mathbf t'} \otimes \mu_{\mathbf t_{N+1}}$, where $\mathbf t' := (t_1, \dots, t_N)$ and $\mathbf t_{N+1} = (t_{N+1}, t_{N+2}, \dots)$.  Define the projections 
$\pi_N(z):= (z_1,\ldots ,z_N)$ for a sequence $z =(z_j)_j\in\mathcal{H}$. Since $p= p \circ \pi_N$ we obtain:  
\begin{align*}
  & \int_{\mathcal{H}} |W_z^{\bf t} p(w)|^2 d\mu_{\mathbf t}(w)=\\
    = &\int_{\mathbb C^{\infty}} |k_{\pi_N(z)}^{\mathbf t'}\circ \pi_N(w)\: p\circ \pi_N(w-z)|^2
    |k_{(I-\pi_N)(z)}^{\mathbf t_{N+1}}\circ (I-\pi_N)(w)|^2~d\mu_{\mathbf t}(w)\\
    =& \int_{\mathbb C^N} |W_{\pi_N(z)}^{\mathbf t'}p(w)|^2~d\mu_{\mathbf t'}(w) \times \underbrace{\int_{\prod_{k=N+1}^\infty \mathbb C} |k_{(I-\pi_N)(z)}^{\mathbf t_{N+1}}(w)|^2~d\mu_{\mathbf t_{N+1}}(w)}_{=1}\\
    =& \int_{\mathcal H} |p(w)|^2 ~d\mu_{\mathbf t}(w).
\end{align*}
Therefore, $W_z^{\mathbf t}p \in L^2(\mathcal H, \mu_{\mathbf t})$ and $W_z^{\mathbf t}$ acts isometrically (say, on analytic polynomials). In finite dimensions, it would now be clear that $W_z^tp \in F_{\mathbf t}^2$, since $W_z^tp$ defines a holomorphic function. In infinite dimensions, this statement is not entirely trivial, as $F_{\mathbf t}^2$ is defined as the closure of the analytic polynomials. Nevertheless, $W_z^tp \in F_{\mathbf t}^2$ remains true, cf. \cite{Janas_Rudol1990}. Hence, $W_z^{\mathbf t}$ extends to an isometric operator on $F_{\mathbf t}^2(\mathcal H)$ for every $z \in \mathcal H_{1}$.
\vspace{1mm}\par 

Further, note that for $(z_k) \subset \mathcal H_{1}^{\mathbf t}$, $z_k \to z$ in $\mathcal H_{1}^{\mathbf t}$ we clearly have
\begin{align*}
    k_{z_k}^{\mathbf t}(w)p(w-z_k) \to k_z^{\mathbf t}(w)p(w-z)
\end{align*}
pointwise almost everywhere on $\ell^2(\mathbb N)$. Hence, {\it Scheff\'{e}'s Lemma} shows that the assignment $z \mapsto W_z^{\mathbf t}p$ is continuous from 
$\mathcal H_1$ to $F_{\mathbf t}^2(\mathcal H)$ for every analytic polynomial $p$. Since these are dense in $F_{\mathbf t}^2(\mathcal H)$, the Weyl operators $W_z^{\mathbf t}$ extend to isometric operators on $F_{\mathbf t}^2(\mathcal{H})$ such $z \mapsto W_z^{\mathbf t}$ is continuous on $\mathcal H_1$ with respect to the strong operator topology. 
\vspace{1mm}\par 
An even stronger statement holds, namely the map $z \mapsto W_z^{\mathbf t}$ is even strongly continuous with respect to the coarser topology of $\mathcal H_{1/2}$, cf. \cite{Janas_Rudol1990}. Therefore, it extends to a map from $\mathcal H_{1/2}$ to the unitary operators on $F_{\mathbf t}^2(\mathcal H)$.

By applying the Weyl operators with $z,w \in \mathcal H_1$ to analytic polynomials, it is not hard to see that 
$$W_z^{\mathbf t} W_w^{\mathbf t} = {e^{-i\sigma_{\mathbf t}(z, w)}}W_{z+w}^{\mathbf t} \hspace{3ex} \mbox{\it and } \hspace{3ex} (W_z^{\mathbf t})^\ast = (W_z^{\mathbf t})^{-1} = W_{-z}^{\mathbf t}.$$
These relations extend to $z,w \in \mathcal H_{1/2}$, as well: 
\begin{lem}[\cite{Janas_Rudol1990, Janas_Rudol1995}]\label{CCR_Weyl_operators_infinitely_many_variables}
The Weyl operators $W_z^{\mathbf t} \in \mathcal L(F_{\mathbf t}^2(\mathcal H))$ depend continuously on $z \in (\mathcal H_{1/2}, \| \cdot \|_{1/2})$ with respect to the 
strong operator topology and
\begin{align*}
    W_z^{\mathbf t} W_w^{\mathbf t} = {e^{-i\sigma_{\mathbf t}(z, w)}}W_{z+w}^{\mathbf t}
\end{align*}
for every $z, w \in \mathcal H_{1/2}$.
\end{lem}

Fix a bounded operator $A \in \mathcal L(F_{\mathbf t}^2(\mathcal H))$ and $z \in \mathcal H_1$. We define the {\it Berezin transform} of $A$ at $z$ by
\begin{align*}
    \widetilde{A}(z) = \big{\langle} Ak_z^{\mathbf t}, k_z^{\mathbf t}\big{\rangle} = \big{\langle} W_{-z}^{\mathbf t} A W_z^{\mathbf t}1, 1\big{\rangle}.
\end{align*}
\begin{lem}\label{lemma_injectivity_BT}
The map $A \mapsto \widetilde{A}$ is injective on $\mathcal L(F_{\mathbf t}^2(\mathcal H))$.
\end{lem}
\begin{proof} If $\widetilde{A} = 0$, then $\widetilde{A}(z) = 0$ for every $z = (z_1, \dots, z_N, 0, 0, \dots)$. We denote by $F_N^2$ the closed linear span in $F_{\mathbf t}^2(\mathcal H)$ 
of the set of all analytic polynomials in the finitely many variables $z_1, \dots, z_N$. Moreover, we write $P_N \in \mathcal L(F_{\mathbf t}^2(\mathcal H))$ for the orthogonal
 projection onto $F_N^2$. Then, $k_z^{\mathbf t}\in F_N^2$ for $z = (z_1, \dots, z_N, 0, 0, \dots)$ and hence:
\begin{align*}
    \widetilde{A}(z) = \big{\langle} A k_z^{\mathbf t}, k_z^{\mathbf t}\big{\rangle} = \big{\langle} P_NA k_z^{\mathbf t}, k_z^{\mathbf t}\big{\rangle}.
\end{align*}
Note that $F_N^2$ can be naturally identified with the $N$-variable Fock-Bargmann space $F_{\mathbf t'}^2=F_{\mathbf t'}^2(\mathbb C^N)$ with weight parameter $\mathbf t' = (t_1, \dots, t_N)$. 
Moreover, 
$$P_N A|_{F_N^2} \in \mathcal L(F_N^2) \cong \mathcal L(F_{\mathbf t'}^2)\hspace{3ex} \mbox{\it and } \hspace{3ex} 0=\widetilde{A}(z) = (P_NA|_{F_N^2})^\sim(z_1, \dots, z_N).$$ 
On the right hand side of the last formula we have use the Berezin transform on $F_{\mathbf t'}^2$. 
Since the Berezin transform is injective on $\mathcal L(F_{\mathbf t'}^2)$ (see \cite{Folland1989}), we conclude that $P_N A|_{F_N^2} = 0$. 
Hence, $P_N A p = 0$ for any analytic polynomial $p$ and $N\in \mathbb{N}$ sufficiently large. Finally, we have
\begin{align*}
    P_N Ap \to Ap, \quad N \to \infty,
\end{align*}
in $F_{\mathbf t}^2(\mathcal H)$, i.e. $Ap = 0$ for analytic polynomials $p$ showing that $A = 0$.
\end{proof}
We denote by $P^{\mathbf t}$ the orthogonal projection from $L^2(\mathcal H, \mu_{\mathbf t})$ onto $F_{\mathbf t}^2(\mathcal H)$. Given $\varphi \in L^\infty(\mathcal H, \mu_{\mathbf t})$ we define the Toeplitz operator $T_{\varphi}^{\mathbf t} \in \mathcal L(F_{\mathbf t}^2(\mathcal H))$ by:
\begin{align*}
    T_{\varphi}^{\mathbf t}(g) = P^{\mathbf t}(\varphi g).
\end{align*}
By applying the injectivity of the Berezin transform in Lemma \ref{lemma_injectivity_BT}, it is not hard to verify that the Weyl operators $W_z^{\mathbf t}$ for $z \in \mathcal H_1$ 
have a representation as Toeplitz operators. More precisely: 
\begin{equation}\label{Weyl_equals_Toeplitz}
    W_z^{\mathbf t} = T_{g_z}^{\mathbf t} \hspace{3ex} \mbox{\it where} \hspace{3ex}  g_z(w) = e^{2i\sigma_{\mathbf t}(w,z) + \frac{1}{2}\| z\|_{1/2}^2}.
\end{equation}
Due to the importance of this fact, we derive an explicit formula for the Berezin transform of both operators: 
\begin{align*}
\widetilde{W_z^{\mathbf t}}(w) &=  e^{-\| w\|_{1/2}^2} \langle W_z^{\mathbf t} K_w^{\mathbf t}, K_w^{\mathbf t}\rangle\\
&= e^{-\| w\|_{1/2}^2-\frac{1}{2}\| z\|_{1/2}^2} \langle e^{\langle \cdot, z\rangle_{1/2}} e^{\langle \cdot - z, w\rangle_{1/2}}, e^{\langle \cdot, w\rangle_{1/2}}\rangle\\
&= e^{-\| w\|_{1/2}^2-\frac{1}{2}\| z\|_{1/2}^2} e^{-\langle z, w\rangle_{1/2}} \langle K_{w+z}^{\mathbf t}, K_w^{\mathbf t}\rangle\\
&= e^{-\| w\|_{1/2}^2-\frac{1}{2}\| z\|_{1/2}^2} e^{-\langle z, w\rangle_{1/2}+\langle w, w+z\rangle_{1/2}},\\
\widetilde{g_z}^{(\mathbf t)}(w) &= e^{\frac{1}{2}\| z\|_{1/2}^2 - \| w\|_{1/2}^2} \langle e^{2i\sigma_{\mathbf t}(\cdot, z)} K_w^{\mathbf t}, K_w^{\mathbf t}\rangle\\
&= e^{\frac{1}{2}\| z\|_{1/2}^2 - \| w\|_{1/2}^2} \langle e^{\langle \cdot, z\rangle_{1/2}} K_w^{\mathbf t}, e^{-\langle \cdot, z\rangle_{1/2}}K_w^{\mathbf t}\rangle\\
&= e^{\frac{1}{2}\| z\|_{1/2}^2 - \| w\|_{1/2}^2} \langle K_{w+z}^{\mathbf t}, K_{w-z}^{\mathbf t}\rangle\\
&= e^{\frac{1}{2}\| z\|_{1/2}^2 - \| w\|_{1/2}^2} e^{\langle w-z, w+z\rangle_{1/2}}
= \widetilde{W_z^{\mathbf t}}(w).
\end{align*}
\par
It is an important but non-trivial fact that $(\ref{Weyl_equals_Toeplitz})$ extends to $z \in \mathcal H_{1/2}$. First, we have to explain in what sense the 
map 
$w \mapsto \langle w, z\rangle_{1/2}$ defines a measurable function on $\mathcal{H}$ for  $z \in \mathcal H_{1/2}$. Clearly, the expression $\langle \cdot ,z\rangle$ is pointwise 
well-defined on $\mathcal H_{1/2}$. However, since $\mathcal H_{1/2}$ is a set of measure zero, this is of no big help. We need a preliminary result: 
\begin{lem}\label{lem:l2norm}
$\| \langle \cdot, z\rangle_{\frac{1}{2}} \|^2 = 2\| z\|_{1/2}^2$ for $z \in \mathcal H_1$. 
\end{lem}
\begin{proof}
We introduce a complex parameter $\lambda \in \mathbb C$. Then, it is not hard to verify the following equalities through standard results on differentiation of parameter integrals, where we use that $\langle \cdot, z\rangle \in L^2(\mathcal H, \mu_{\mathbf t})$ (this is easily established):
\begin{multline*}
\frac{\partial^2}{\partial \overline{\lambda}\partial \lambda} \int_{\mathcal H} e^{i 2\re \langle w, \lambda z\rangle_{1/2}}~d\mu_{\mathbf t}(w) = 
\frac{\partial^2}{\partial \overline{\lambda}\partial \lambda} \int_{\mathcal H} e^{i\overline{\lambda} \langle w, z\rangle_{1/2}+i\lambda \langle z, w\rangle_{1/2}}~d\mu_{\mathbf t}(w)\\
= -\int_{\mathcal H} |\langle z, w\rangle_{1/2}|^2 e^{i\overline{\lambda} \langle w, z\rangle_{1/2}+i\lambda \langle z, w\rangle_{1/2}}~d\mu_{\mathbf t}(w).
\end{multline*}
On the other hand, using Equality (\ref{Gaussmeasure_Fouriertrafo}):
\begin{align*}
\frac{\partial^2}{\partial \overline{\lambda}\partial \lambda} 
\int_{\mathcal H} e^{i 2\re \langle w, \lambda z\rangle_{1/2}}~d\mu_{\mathbf t}(w) 
&=\frac{\partial^2}{\partial \overline{\lambda}\partial \lambda}  \int_{\mathcal H} e^{i \re\langle w, 2\lambda B^{-1}z\rangle_{\ell^2}}~d\mu_{\mathbf t}(w)\\
&=\frac{\partial^2}{\partial \overline{\lambda}\partial \lambda}  e^{-\frac{1}{2} \langle 2\lambda BB^{-1}z, 2\lambda B^{-1}z\rangle_{\ell^2}}\\
&= \frac{\partial^2}{\partial \overline{\lambda}\partial \lambda}  e^{-2 \lambda \overline{\lambda} \langle z, B^{-1}z\rangle_{\ell^2}}\\
&=\frac{\partial^2}{\partial \overline{\lambda}\partial \lambda}  e^{-2\lambda \overline{\lambda} \langle z, z\rangle_{1/2}}\\
& =\Big{(} -2\| z\|_{1/2}^2  + 4 |\lambda|^2 \| z\|_{1/2}^4\Big{)}  e^{-2\lambda \overline{\lambda} \| z\|_{1/2}^2}.
\end{align*}
We therefore obtain:

\begin{equation*}
\int_{\mathcal H} |\langle z, w\rangle_{1/2}|^2 e^{i 2\re \langle w, \lambda z\rangle_{1/2}}~d\mu_{\mathbf t}(w) 
= \big{(}2-4|\lambda|^2 \|z\|_{1/2}^2\big{)}\| z\|_{1/2}^2 e^{-2|\lambda|^2 \| z\|_{1/2}^2}. 
\end{equation*}
Choosing $\lambda = 0$ yields the desired equality.
\end{proof}
Hence, if $(z_k) \subset \mathcal H_1$ is a Cauchy sequence with respect to $\| \cdot\|_{1/2}$, then $\langle \cdot, z_k\rangle_{1/2}$ is Cauchy in $L^2(\mathcal H, \mu_{\mathbf t})$. 
In conclusion, the family of functions $\langle \cdot, z\rangle_{1/2} \in L^2(\mathcal H, \mu_{\mathbf t})$ continuously extends to $z \in \mathcal H_{1/2}$. In particular, for each $z \in \mathcal H_{1/2}$, $\langle w, z\rangle_{1/2}$ is a well-defined expression for almost every $w \in \mathcal H$ and, as a function of $w$, measurable. In particular,
\begin{align*}
g_z(w) = e^{2i\sigma_{\mathbf t}(w, z) + \frac{1}{2}\| z\|_{1/2}^2} = e^{\langle w, z\rangle_{1/2}} e^{-\langle z, w\rangle_{1/2}} e^{\frac{1}{2}\|z\|_{1/2}^2}
\end{align*}
is an almost everywhere well-defined function  with $$\| g_z\|_{\infty} \leq e^{\frac 12 \|z\|_{1/2}^2}, \hspace{4ex} z\in \mathcal{H}_{1/2}.$$
We also note that, since $\langle \cdot, z\rangle_{1/2} \in F_{\mathbf t}^2(\mathcal H)$ for every $z \in \mathcal H_{1}$, the same is true for $z \in \mathcal H_{1/2}$. Iterating 
the procedure from the proof of Lemma \ref{lem:l2norm}, one easily obtains:
\begin{align*}
\| \langle \cdot, z\rangle_{1/2}^k\|^2 = 2^k \| z\|_{1/2}^{2k} \hspace{3ex} \mbox{\it for } \hspace{3ex} k \in \mathbb{N}. 
\end{align*}
Next, we calculate: 
\begin{align*}
\sum_{k=0}^\infty \frac{\| \langle \cdot, z\rangle_{1/2}^k\|}{k!} = \sum_{k=0}^\infty \frac{2^{\frac{k}{2}} \| z\|_{1/2}^{k}}{k!} = e^{\sqrt{2}\| z\|_{1/2}}. 
\end{align*}
We obtain: 
\begin{align*}
\sum_{k=0}^\infty \frac{\langle \cdot, z\rangle_{1/2}^k}{k!} = e^{\langle \cdot, z\rangle_{1/2}} = K_z^{\mathbf t} \in F_{\mathbf t}^2(\mathcal H)
\end{align*}
for every $z\in \mathcal H_{1/2}$. By a standard density argument, we have: 
\begin{align*}
\langle K_z, K_w\rangle = e^{\langle w, z\rangle_{1/2}} \hspace{3ex} \mbox{\it for} \hspace{3ex} z, w \in \mathcal H_{1/2}. 
\end{align*}
Now, one can compute the Berezin transform of $g_z$ as above. For every $z \in \mathcal H_{1/2}$ and $w \in \mathcal H_1$ one obtains: 
\begin{align*}
\widetilde{T_{g_z}^{\mathbf t}}(w) = e^{\frac{1}{2}\| z\|_{1/2}^2 - \| w\|_{1/2}^2} e^{\langle w-z, w+z\rangle_{1/2}}. 
\end{align*}
Further, since $W_z^{\mathbf t}$ continuously (in strong operator topology) depends on the parameter $z \in \mathcal H_{1/2}$, we conclude that the Berezin transform 
$\widetilde{W_z^{\mathbf t}}$ continuously (in the topology of pointwise convergence) depends on $z \in \mathcal H_{1/2}$, which gives
\begin{align*}
\widetilde{W_z^{\mathbf t}}(w) = e^{\frac{1}{2}\| z\|_{1/2}^2 - \| w\|_{1/2}^2} e^{\langle w-z, w+z\rangle_{1/2}}
\end{align*}
for every $z \in \mathcal H_{1/2}$ and $w \in \mathcal H_1$. Comparing Berezin transforms, we have shown that $T_{g_z}^{\mathbf t} = W_z^{\mathbf t}$ extends to $z \in \mathcal H_{1/2}$.
\vspace{1ex}\par
In an abuse of notation, we will write $\mathcal R(\mathcal H_{1/2}, \sigma_{\mathbf t})$ for the representation of the resolvent algebra on $F_{\mathbf t}^2(\mathcal H)$ with 
respect to the symplectic space $(\mathcal H_{1/2}, \sigma_{\bf t})$. Since the Weyl operators $W_z^{\mathbf t}$, $z \in \mathcal H_{1/2}$ satisfy the CCR in 
Lemma \ref{CCR_Weyl_operators_infinitely_many_variables}, we start again by expressing the resolvent $R(\lambda, z)$ in form of a Laplace transform of the one-parameter group 
$(W_{tz})_{t \in \mathbb{R}}$.  
For $\re(\lambda) > 0$ we have: 
\begin{align*}
    R(\lambda, z) = i\int_0^\infty e^{-\lambda s} W_{-sz}^{\mathbf t}~ds,
\end{align*}
which exists as an integral in strong operator topology. There is an analogous formula in the case $\re(\lambda) < 0$. Consider the space of (classical) symbols:
\begin{align*}
    \mathcal R_{cl}^{\mathbf t} = C^\ast \Big{(}\{(\lambda - 2i\sigma_{\mathbf t}(\cdot, z))^{-1}: \lambda \in \mathbb C \setminus i \mathbb R, ~z \in \mathcal H_{1/2}\}\Big{)}.
\end{align*}
Moreover, we write $\mathcal T^{\mathbf t}(\mathcal R_{cl}^{\mathbf t})$ for the C$^\ast$ algebra generated by Toeplitz operators over $F_{\mathbf t}^2(\mathcal H)$ with symbols in 
$\mathcal R_{cl}^{\mathbf t}$. 
\begin{thm}\label{Theorem_Last_section}
The following inclusion holds true: $\mathcal R(\mathcal H_{1/2}, \sigma_{\mathbf t}) \subset \mathcal T^{\mathbf t}(\mathcal R_{cl}^{\mathbf t})$.
\end{thm}
\begin{proof}
Without loss of generality we assume $\re(\lambda) > 0$. We will verify that the representation of the resolvent in form of a Laplace transform of 
Weyl operators defines an element in the Toeplitz algebra. The following integrals are to be understood as improper Riemann integrals in strong operator topology:
\begin{align*}
    \int_0^\infty e^{-\lambda s} W_{-sz}^{\mathbf t}~ds &= \int_0^\infty e^{-\lambda s} T_{g_{-sz}}^{\mathbf t}~ds\\
    &= \int_0^\infty e^{-\lambda s} e^{\frac{s^2}{2}\| z\|_{1/2}^2} T_{\exp(-2is\sigma_{\mathbf t}(\cdot, z))}^{\mathbf t}~ds\\
    &= \int_0^\infty e^{-\lambda s} \sum_{k=0}^\infty \frac{s^{2k}\| z\|_{1/2}^{2k}}{2^k k!} T_{\exp(-2is\sigma_{\mathbf t}(\cdot, z))}^{\mathbf t}~ds. 
\end{align*}
Since the Weyl operators are unitary, and hence satisfy $\| W_z^{\mathbf t}\| = 1$, it follows that: 
\begin{align*}
\| T_{\exp(-2is\sigma_{\mathbf t}(\cdot, z))}^{\mathbf t}\| = e^{-\frac{s^2}{2}\| z\|_{1/2}^2}.
\end{align*}
Therefore, the dominated convergence theorem gives as $m \rightarrow \infty$: 
\begin{multline*}
    \Big{\|} iR(\lambda, z) - \sum_{k=0}^m \frac{\| z\|_{1/2}^{2k}}{2^k k!} \int_0^\infty s^{2k} e^{-\lambda s} T_{\exp(-2is\sigma_{\mathbf t}(\cdot, z))}^{\mathbf t}~ds\Big{\|} \\
    \leq \int_0^\infty e^{-\lambda s} \left|  e^{\frac{s^2}{2}\| z\|_{1/2}^2} - \sum_{k=0}^{m} \frac{s^{2k}\| z\|_{1/2}^{2k}}{k! 2^k} \right| 
    e^{-\frac{s^2}{2}\| z\|_{1/2}^2}~ds
   \to 0, \quad m \to \infty.
\end{multline*}
We have therefore seen that
\begin{align*}
    iR(\lambda, z) = \sum_{k=0}^\infty \frac{1}{k!} \frac{\| z\|_{1/2}^{2k}}{2^k}  \int_0^\infty T_{s^{2k}\exp(-\lambda s + 2is\sigma_{\mathbf t}(\cdot, z))}^{\mathbf t}~ds, 
\end{align*}
where the series converges in operator norm. Fix for the moment $N > 0$. Since the Riemann integral $\int_0^N T_{s^{2k}\exp(-\lambda s + 2is\sigma_{\mathbf t}(\cdot, z))}^{\mathbf t}~ds$ exists as a limit of Riemann sums in strong operator topology (this easily follows from the fact that the mapping $s \mapsto W_{sz}^{\mathbf t}$ is strongly continuous), we obtain for the Berezin transform:

\begin{multline*}
   \left( \int_0^N T_{s^{2k}\exp(-\lambda s + 2is\sigma_{\mathbf t}(\cdot, z))}^{\mathbf t}~ds\right)^{\sim}(u)
    =\Big{\langle}  T_{ \int_0^Ns^{2k}\exp(-\lambda s + 2is\sigma_{\mathbf t}(\cdot, z))ds}^{\mathbf t} ~k_u^{\mathbf t}, k_u^{\mathbf t}\Big{\rangle}.
\end{multline*}

Therefore, injectivity of the Berezin transform (Lemma \ref{lemma_injectivity_BT}) shows:
\begin{align*}
  \int_0^N T_{s^{2k}\exp(-\lambda s + 2is\sigma_{\mathbf t}(\cdot, z))}^{\mathbf t} ds = T_{\int_0^N s^{2k}\exp(-\lambda s + 2is\sigma_{\mathbf t}(\cdot, z))ds}^{\mathbf t}.
\end{align*}
Since
\begin{align*}
    \int_0^N s^{2k}\exp(-\lambda s + 2is\sigma_{\mathbf t}(\cdot, z))ds \overset{N \rightarrow \infty}{\longrightarrow }\int_0^\infty s^{2k}\exp(-\lambda s + 2is\sigma_{\mathbf t}(\cdot, z))ds, 
\end{align*}
uniformly, this gives
\begin{align*}
    \int_0^\infty T_{s^{2k}\exp(-\lambda s + 2is\sigma_{\mathbf t}(\cdot, z))}^{\mathbf t} ~ds = T_{\int_0^\infty {s^{2k}\exp(-\lambda s + 2is\sigma_{\mathbf t}(\cdot, z))}~ds}^{\mathbf t}.
\end{align*}
For the symbol, standard facts on the Laplace transform yield
\begin{align*}
    \int_0^\infty {s^{2k}e^{-\lambda s + 2is\sigma_{\mathbf t}(w, z)}}~ds = (2k)!(\lambda - 2i\sigma_{\mathbf t}(w, z))^{-(2k+1)}.
\end{align*}
This is now, as a function of $w \in \mathcal H$, bounded and measurable, which finishes the proof.
\end{proof}

The careful reader may have noticed that the proofs of Proposition \ref{formula:resolvent} and Lemma \ref{Lemma_Berezin_transform_of_the_resolvent} generalize to the infinite dimensional phase space, i.e. analogous formulas for the Berezin transforms of products of resolvents and the classical resolvent functions, respectively, are valid. Nevertheless, since a {\it correspondence theorem} in the infinite dimensional setup (similar to Theorem \ref{Correspondence_Theorem}) is not available at the moment, this observation is of no further use for proving a refinement of Theorem \ref{Theorem_Last_section}. The biggest issue is the lack of a Haar measure on the infinite dimensional symplectic space. In the finite dimensional framework of $\mathbb{C}^n$, this measure coincides with the Lebesgue measure and is at the heart of \emph{quantum harmonic analysis}. An important ingredient to the theory is a correspondence between the spaces 
$L^1(\mathbb C^n)$ and $\mathcal T^1(F_{\mathbf t}^2)$ (the latter being the space of trace class operators). It is not clear what an appropriate interpretation of ``$L^1(\mathcal H_{1/2})$'' 
should be in order to develop a \emph{quantum harmonic analysis} and \emph{correspondence theory} for infinite dimensional phase spaces. This will be a key issue for our future work. 
\nocite{Janas_Rudol1990, Janas_Rudol1995, Wick_Wu2022}

\bibliographystyle{amsplain}
\bibliography{References}

\providecommand{\bysame}{\leavevmode\hbox to3em{\hrulefill}\thinspace}
\providecommand{\MR}{\relax\ifhmode\unskip\space\fi MR }
\providecommand{\MRhref}[2]{%
  \href{http://www.ams.org/mathscinet-getitem?mr=#1}{#2}
}
\providecommand{\href}[2]{#2}
\begin{thebibliography}{10}

\bibitem{Bauer_Isralowitz2012}
W.~Bauer and J.~Isralowitz, \emph{{C}ompactness characterization of operators
  in the {T}oeplitz algebra of the {F}ock space {$F_\alpha^p$}}, J. Funct.
  Anal. \textbf{263} (2012), 1323--1355.

\bibitem{Bauer_Rodriguez}
W.~Bauer and M.~A. Rodriguez~Rodriguez, \emph{{Commutative Toeplitz algebras
  and their Gelfand theory: old and new results}}, Complex Anal. Oper. Theory
  \textbf{16} (2022), 77.

\bibitem{Berger_Coburn1987}
C.~L. Berger and L.~A. Coburn, \emph{{Toeplitz operators on the Segal-Bargmann
  space}}, Trans. Amer. Math. Soc. \textbf{301} (1987), 813--829.

\bibitem{Bourbaki}
N.~Bourbaki, \emph{Elements of mathematics: general topology chapters 1 - 4},
  Springer, 1989.

\bibitem{Bratteli_Robinson2}
O.~Bratteli and D.~W. Robinson, \emph{{Operator algebras and quantum
  statistical mechanics 2}}, 2nd ed., Springer-Verlag, 1997.

\bibitem{Buchholz2014}
D.~Buchholz, \emph{The resolvent algebra: Ideals and dimension}, J. Funct.
  Anal. \textbf{266} (2014), 3286--3302.

\bibitem{Buchholz_Grundling2008}
D.~Buchholz and H.~Grundling, \emph{{The resolvent algebra: A new approach to
  canonical quantum systems}}, J. Funct. Anal. \textbf{254} (2008), 2725--2779.

\bibitem{Coburn1999}
L.~A. Coburn, \emph{{The measure algebra of the Heisenberg group}}, J. Funct.
  Anal. \textbf{161} (1999), 509--525.

\bibitem{DaPrato}
G.~Da~Prato, \emph{{An introduction to infinite-dimensional analysis}},
  Universitext, Springer Verlag, 2006.

\bibitem{Engel_Nagel}
K.-J. Engel and R.~Nagel, \emph{{One-parameter semigroups for linear evolution
  equations}}, Graduate texts in mathematics, vol. 194, Springer-Verlag, 2000.

\bibitem{Fannes_Verbeure1974}
M.~Fannes and A.~Verbeure, \emph{{On the time evolution automorphisms of the
  CCR-algebra for quantum mechanics}}, Comm. Math. Phys. \textbf{35} (1974),
  257--264.

\bibitem{Folland1989}
G.~B. Folland, \emph{{Harmonic analysis in phase space}}, Princeton University
  Press, 1989.

\bibitem{Fulsche2020}
R.~Fulsche, \emph{{Correspondence theory on $p$-Fock spaces with applications
  to Toeplitz algebras}}, J. Funct. Anal. \textbf{279} (2020), 108661.

\bibitem{Fulsche2021}
\bysame, \emph{{Toeplitz operators on non-reflexive Fock spaces}}, preprint
  available under arXiv:2202.11440, 2021.

\bibitem{Janas_Rudol1990}
J.~Janas and K.~Rudol, \emph{{Toeplitz operators on the Segal-Bargmann space of
  infinitely many variables}}, {Linear Operators in Function spaces. 12th
  International Conference on Operator Theory Timi\c{s}oara (Romania) June 6-16
  1988} (H.~Helson, B.~Sz.-Nagy, and F.-H. Vasilescu, eds.), Operator Theory:
  Advances and Applications, no.~43, Birkh"auser Verlag, 1990, pp.~217--228.

\bibitem{Janas_Rudol1995}
\bysame, \emph{{Toeplitz operators in infinitely many variables}}, {Topics in
  operator theory, operator algebras and applications. 15th international
  conference on operator theory, Timişoara, Romania, June 6–10, 1994.}, IMAR
  Bucharest, 1995, pp.~147--160.

\bibitem{Reed1}
M.~Reed and B.~Simon, \emph{{Methods of modern mathematical physics 1:
  functional analysis}}, Academic Press, 1972.

\bibitem{Skorohod}
A.~V. Skorohod, \emph{{Integration in Hilbert space}}, Ergebnisse der
  Mathematik und ihrer Grenzgebiete, Springer Verlag, 1974.

\bibitem{vanNuland}
T.~D.~H. van Nuland, \emph{Quantization and the resolvent algebra}, J. Funct.
  Anal. \textbf{277} (2019), 2815--2838.

\bibitem{Werner1984}
R.~Werner, \emph{{Quantum harmonic analysis on phase space}}, J. Math. Phys.
  \textbf{25} (1984), 1404--1411.

\bibitem{Wick_Wu2022}
B.~W. Wick and S.~Wu, \emph{{Fock space on $\mathbb C^\infty$ and Bose-Fock
  space}}, J. Math. Anal. Appl. \textbf{505} (2022), 125499.

\end{thebibliography}

\end{document}